\documentclass[11pt,a4paper]{article}
\usepackage[T1]{fontenc}
\usepackage[latin9]{luainputenc}
\usepackage{geometry}
\usepackage{dsfont}
\usepackage{color}
\usepackage{mathrsfs}
\usepackage{amsfonts}
\usepackage{amsmath}
\usepackage{amsthm}
\usepackage{amssymb}
\usepackage{graphicx}

\usepackage{tikz}
\usetikzlibrary{decorations.pathreplacing,decorations.markings}
\usetikzlibrary{arrows}

\tikzset{middlearrow/.style={
        decoration={markings,
            mark= at position 0.5 with {\arrow{#1}} ,
        },
        postaction={decorate}
    }
}

\tikzset{
  on each segment/.style={
    decorate,
    decoration={
      show path construction,
      moveto code={},
      lineto code={
        \path [#1]
        (\tikzinputsegmentfirst) -- (\tikzinputsegmentlast);
      },
      curveto code={
        \path [#1] (\tikzinputsegmentfirst)
        .. controls
        (\tikzinputsegmentsupporta) and (\tikzinputsegmentsupportb)
        ..
        (\tikzinputsegmentlast);
      },
      closepath code={
        \path [#1]
        (\tikzinputsegmentfirst) -- (\tikzinputsegmentlast);
      },
    },
  },
  mid arrow/.style={postaction={decorate,decoration={
        markings,
        mark=at position .5 with {\arrow[#1]{stealth}}
      }}},
}

\usepackage{float}

\usepackage[unicode=true,
 bookmarks=true,bookmarksnumbered=true,bookmarksopen=false,
 breaklinks=false,pdfborder={0 0 1},backref=false,colorlinks=true]
 {hyperref}
\hypersetup{
 linkcolor=blue, citecolor=black, urlcolor=blue, filecolor=blue, pdfpagelayout=OneColumn, pdfnewwindow=true, pdfstartview=XYZ, plainpages=false}

\makeatletter
\numberwithin{equation}{section}
\theoremstyle{plain}
\newtheorem{thm}{\protect\theoremname}[section]
  \theoremstyle{remark}
  \newtheorem{rem}[thm]{\protect\remarkname}
  \theoremstyle{definition}
  \newtheorem{defn}[thm]{\protect\definitionname}
  \theoremstyle{plain}
  \newtheorem{lem}[thm]{\protect\lemmaname}
  \theoremstyle{plain}
  \newtheorem{prop}[thm]{\protect\propositionname}
  \theoremstyle{remark}
  \newtheorem*{rem*}{\protect\remarkname}

\usepackage{ifpdf} % part of the hyperref bundle
\ifpdf % if pdflatex is used

 \IfFileExists{lmodern.sty}{\usepackage{lmodern}}{}

\fi % end if pdflatex is used

\usepackage{fancyhdr}
\date{ }

\def\definitionname{Definition}
\def\theoremname{Theorem}
\def\propositionname{Proposition}
\def\lemmaname{Lemma}

\def\remarkname{Remark}

\makeatother

  \providecommand{\definitionname}{Definition}
  \providecommand{\lemmaname}{Lemma}
  \providecommand{\propositionname}{Proposition}
  \providecommand{\remarkname}{Remark}
\providecommand{\theoremname}{Theorem}
\newcommand{\N}{\mathbb N}
\newcommand{\R}{\mathbb R}
\newcommand{\cA}{\mathcal A}
\newcommand{\cL}{\mathcal L}
\newcommand{\cV}{\mathcal V}

\def\ds{\displaystyle}

\begin{document}
\date{\today}

\title{A Class of  Infinite Horizon Mean Field Games on Networks}

\author{Yves Achdou \thanks { Univ. Paris Diderot, Sorbonne Paris Cit{\'e}, Laboratoire Jacques-Louis Lions, UMR 7598, UPMC, CNRS, F-75205 Paris, France.
 achdou@ljll.univ-paris-diderot.fr},
Manh-Khang Dao \thanks {IRMAR, Universit{\'e} de Rennes 1, Rennes, France},
Olivier Ley \thanks {IRMAR, INSA Rennes, France},
Nicoletta Tchou \thanks {IRMAR, Universit{\'e} de Rennes 1, Rennes, France, nicoletta.tchou@univ-rennes1.fr}
}
\maketitle

\abstract{ We consider  stochastic mean field games for which the state space is a network.
In the ergodic case, they are described by a system coupling  a Hamilton-Jacobi-Bellman equation and a Fokker-Planck equation,
 whose unknowns are the invariant measure $m$, a value function $u$, and the ergodic constant $\rho$. 
The  function $u$ is continuous and satisfies general Kirchhoff conditions at the vertices. 
The invariant measure $m$ satisfies dual transmission conditions:
 in particular, $m$ is discontinuous across the vertices in general, and the values of $m$ on each side of the vertices satisfy special compatibility conditions. 
Existence and uniqueness are proven, under suitable assumptions.
}

\section{\label{sec: Mean-field-games}Introduction and main results}

Recently, an important research activity on mean field games (MFGs for short) has been initiated since the pioneering works \cite{LL2006-A,LL2006-B,LL2007} of  Lasry and Lions (related ideas have been developed independently in the engineering
literature  by Huang-Caines-Malham{\'e}, see for example \cite{HMC2006,HCM2007-A,HMC2007-B}): it aims at studying the  asymptotic behavior of  stochastic differential games (Nash equilibria) as the number $N$ of agents
tends to infinity.
  In  these models, it is assumed that the agents are all identical and that  an individual agent can hardly influence the outcome of the game. 
 Moreover, each individual strategy is influenced by  some averages of functions  of 
the states of the other agents. In the limit when $N\to +\infty$, a given agent feels the presence of the others  through the 
statistical distribution of the states. Since perturbations of the strategy of a single agent do not influence the statistical states distribution, the latter acts as a parameter  in the control problem to be solved by each  agent.  The delicate question of the passage to the limit is one of the main topics of the  book of Carmona and Delarue, \cite{CD2017}.  When the dynamics of the agents are independent stochastic processes, MFGs
naturally lead to a coupled system of two partial differential equations (PDEs for short), a forward in time Kolmogorov or Fokker-Planck (FP) equation  and a backward Hamilton-Jacobi-Bellman (HJB) equation. The unknown of this system is a pair of functions: the value function of the stochastic optimal control problem solved by a representative agent and the density of the distribution of states. In the infinite horizon limit, one obtains a system of two stationary PDEs.

A very nice introduction to the theory of MFGs is supplied in the notes of Cardaliaguet \cite{cardaliaguet2010}. Theoretical results on the existence of classical solutions to the previously mentioned system of PDEs can be found in \cite{LL2006-A,LL2006-B,LL2007,GPV2014,GPS2015, GP2016-A}.  Weak solutions have been studied in \cite{LL2007,Parma,Pumi,AP2016}. The  numerical approximation of these systems of PDEs has been discussed in \cite{MR2679575,Achdou2013,AP2016}. 

A network (or a graph) is a set of items, referred to as vertices (or nodes/crosspoints), with connections between them referred to as edges.    In the recent years, there has been an increasing interest  in the investigation of dynamical systems and differential equations on networks, in particular in connection with problems of data transmission and traffic management
 (see for example \cite{ MR2328174, MR2448938, MR3640560}). The literature  on optimal control in which the state variable takes its values on a network is recent:  deterministic control problems and related Hamilton-Jacobi equations were studied in   \cite{MR3057137,MR3023064,MR3358634,MR3621434,MR3556345,MR3729588}. Stochastic processes on networks and related Kirchhoff conditions at the vertices  were studied in \cite{FW1993,FS2000}.\\

 The present work is devoted to  infinite horizon stochastic mean field games taking place on networks. 
The most important difficulty will be to deal with the transition conditions at the vertices. The latter  are obtained 
 from the theory of  stochastic control in \cite{FW1993,FS2000}, see Section \ref{subsec: A derivation of the MFG system} below.
In \cite{CM2016}, the first article on MFGs on networks, Camilli and Marchi consider a particular type of Kirchhoff condition  at the vertices for the value function: this condition comes from an assumption which can be informally stated as follows: consider a vertex $\nu$ of the network and 
assume that it is the intersection of $p$ edges $\Gamma_{1},\dots, \Gamma_{p} $, ;
 if, at time $\tau$, the controlled stochastic process $X_t$ associated to a given agent hits $\nu$, 
then the probability that $X_{\tau^+}$ belongs to $\Gamma_i$ is proportional to the diffusion coefficient in $\Gamma_i$. Under this assumption, it can be seen that the density of the distribution of states is continuous at the vertices of the network.  
In the present work, the above mentioned  assumption is not made any longer.  Therefore, it will be seen below that the value function  satisfies more general Kirchhoff conditions,
 and accordingly, that the density of the distribution of states is no longer continuous at the vertices; the continuity condition is then replaced by suitable
 compatibility conditions on the jumps across the vertex.    Moreover, as it will be explained in Remark~\ref{sec:main-result-2} below, more general assumptions on the coupling costs will be made. Mean field games on networks with finite horizon will be considered in a forthcoming paper.
 
After obtaining the  transmission conditions at the vertices for both the value function and the density,
we shall prove existence and uniqueness of weak solutions  of the uncoupled HJB and FP equations (in suitable Sobolev spaces). 
We have chosen to work with weak solutions because it is a convenient way to deal with existence and uniqueness in the stationary regime,
 but also because it is difficult to avoid it in the nonstationary case, see the forthcoming work on finite horizon MFGs.   
  Classical arguments will then lead to the regularity of the solutions.
 Next, we shall establish the existence result for the MFG system by a fixed point argument and a truncation technique.
 Uniqueness will also be  proved under suitable assumptions.

The present work is organized as follows: the remainder of Section~\ref{sec: Mean-field-games} is devoted to
setting the problem and obtaining the system of differential equations and the transmission conditions at the vertices. Section~\ref{sec: linear problems} contains useful results,
 first about some linear boundary value problems  with elliptic equations, then on a pair of linear Kolmogorov and Fokker-Planck equations in duality.
By and large, the existence of weak solutions is obtained by applying Banach-Necas-Babu\v ska theorem to a special pair of Sobolev spaces referred to as $V$ and $W$ below and Fredholm's alternative, and uniqueness comes from a maximum principle. Section~\ref{sec: H-J equation and ergodic problem} is devoted to 
the HJB equation associated with an ergodic problem. Finally,  the proofs of the main results of existence and uniqueness for the MFG system of differential equations are completed in Section~\ref{sec: Mean-field-games}.

\subsection{Networks and function spaces}

\subsubsection{The geometry}\label{sec:geometry}
A bounded network $\Gamma$ (or a bounded connected graph) is a connected subset of $\R ^n$ made of a finite number of  bounded non-intersecting straight segments, referred to as edges, which connect  
nodes referred to as vertices. The finite collection of vertices and the finite set of closed edges are respectively denoted by $\mathcal{V}:=\left\{ \nu_{i}, i\in I\right\}$
and  $\mathcal{E}:=\left\{ \Gamma_{\alpha}, \alpha\in\mathcal{A}\right\}$, 
where $I$ and $A$ are finite sets of indices contained in $\N$.  We assume that for $\alpha,\beta \in \cA$,  if $\alpha\not=\beta$, 
 then $\Gamma_\alpha\cap \Gamma_\beta$ is either empty or made of  a single vertex.
 The length of $\Gamma_{\alpha}$ is denoted by $\ell_{\alpha}$.
Given $\nu_{i}\in\mathcal{V}$, the set of indices of edges that are
adjacent to the vertex $\nu_{i}$ is denoted by
 $\mathcal{A}_{i}=\left\{ \alpha\in\mathcal{A}:\nu_{i}\in\Gamma_{\alpha}\right\} $.
A vertex $\nu_{i}$ is named a {\sl boundary vertex} if $\sharp\left(\mathcal{A}_{i}\right)=1$,
otherwise it is named a {\sl transition vertex}. The set containing all
the boundary vertices is named the {\sl boundary} of the network and is 
denoted by $\partial\Gamma$ hereafter.

The edges $\Gamma_{\alpha}\in\mathcal{E}$  are oriented in an arbitrary manner. In most of what follows, we shall make the 
following arbitrary choice that an edge $\Gamma_{\alpha}\in\mathcal{E}$ connecting two vertices $\nu_i$ and $\nu_j$, with $i<j$ is oriented from $\nu_i$ toward $\nu_j$:
% An edge $\Gamma_{\alpha}\in\mathcal{E}$ can be written
%we set $\Gamma_\alpha=[\nu_i,\nu_j]=\{(\ell_\alpha-y)\nu_i+y\nu_j: y\in [0,\ell_\alpha]\}$, and
this induces a natural parametrization $\pi_\alpha: [0,\ell_\alpha]\to \Gamma_\alpha=[\nu_i,\nu_j]$:
\begin{equation}
\pi_\alpha(y)=(\ell_\alpha-y)\nu_i+y\nu_j\quad\text{for } y\in  [0,\ell_\alpha]. \label{eq: parametrization}
\end{equation}
%as well as an orientation of $\Gamma_\alpha$.
For a function 
$v:\Gamma\rightarrow\mathbb{R}$ and $\alpha\in\mathcal{A}$,
we define $v_\alpha: (0,\ell_\alpha)\rightarrow\mathbb{R}$ by
\begin{displaymath}
v_{\alpha} (x):=v|_{\Gamma_{\alpha}}\circ\pi_{\alpha} (x),\quad \hbox{ for all }x\in (0,\ell_\alpha).  
\end{displaymath}
% Notice that $v_\alpha$ is a priori defined inside the edge but can be extended  by continuity in $[0,\ell_{\alpha}]$
%  when it is possible (see below).

\begin{rem}\label{rem:new network}
In what precedes, the edges have been arbitrarily  oriented from the vertex with the smaller index toward the vertex with the larger one.
 Other choices are of course possible.
 In particular, by possibly dividing a single edge into two, adding thereby new artificial vertices, it is always possible to assume that for all vertices $\nu_i\in\mathcal{V}$,
\begin{equation}
\text{either }\pi_{\alpha}(\nu_i)=0,\text{ for all }\alpha\in\mathcal{A}_i\text{ or }\pi_{\alpha}(\nu_i)=\ell_\alpha,\text{ for all }\alpha\in\mathcal{A}_i.\label{oriented networks}
\end{equation} 
This idea was used by  Von Below in \cite{Below1988}: some edges of $\Gamma$ are cut into two by adding artificial vertices
so that the new oriented network $\overline{\Gamma}$ has the property \eqref{oriented networks}, see Figure \ref{fig: Network intro} for an example.

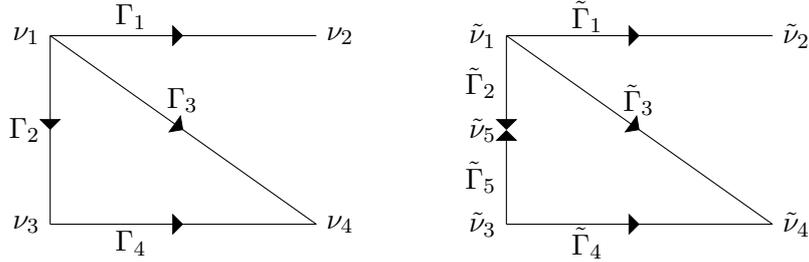
\begin{figure}[H]
  \begin{center}
    \begin{tikzpicture}[scale=0.5, trans/.style={thick,<->,shorten >=2pt,shorten <=2pt,>=stealth} ]
     % \path [draw=blue,postaction={on each segment={mid arrow=red}}]
      \draw (0.8,5.5)node[left] {$\Gamma_1$};
      \draw[middlearrow={triangle 90}] (-2,5) node [left]{$\nu_1$} -- (5,5)  node [right]{$\nu_2$};
      \draw (-2,2.5)node[left] {$\Gamma_2$};
      \draw[middlearrow={triangle 90}] (-2,5) -- (-2,0) node [left]{$\nu_3$};
      \draw (0.8,3.3)node[right] {$\Gamma_3$};
      \draw[middlearrow={triangle 90}] (-2,5) -- (5,0) node [right]{$\nu_4$};
      \draw (0.8,-0.5)node[left] {$\Gamma_4$};
      \draw[middlearrow={triangle 90}] (-2,0) -- (5,0);

      \draw (12.8,5.5)node[left] {$\tilde\Gamma_1$};
      \draw[middlearrow={triangle 90}] (10,5) node [left]{$\tilde \nu_1$} -- (17,5) node [right]{$\tilde\nu_2$};
      \draw (10,3.75)node[left] {$\tilde \Gamma_2$};
      \draw[-{triangle 90}] (10,5) -- (10,2.5) node [left]{$\tilde \nu_5$};
       \draw (10,1.25)node[left] {$\tilde \Gamma_5$};
      \draw[-{triangle 90}] (10,0)  node [left]{$\tilde \nu_3$} -- (10,2.5) ;
      \draw (12.8,3.3)node[right] {$\tilde\Gamma_3$};
      \draw[middlearrow={triangle 90}] (10,5) -- (17,0) node [right]{$\tilde \nu_4$};
      \draw (12.8,-0.5)node[left] {$\tilde\Gamma_4$};
      \draw[middlearrow={triangle 90}] (10,0) -- (17,0);
    \end{tikzpicture}
    \caption{Left: the network $\Gamma$ in which the edges are oriented  toward the vertex with larger index ($4$ vertices and $4$ edges). Right: a new network $\tilde \Gamma$ obtained by adding an artificial vertex ($5$ vertices and $5$ edges): the oriented edges sharing a given vertex $\nu$ either have all their starting point equal $\nu$, or have all  their terminal point equal $\nu$.}
   \label{fig: Network intro}
  \end{center}
\end{figure}

In Sections~\ref{sec:class-stoch-proc} and  \ref{subsec: A derivation of the MFG system} below, especially when dealing with stochastic calculus, it will be convenient to assume that property \eqref{oriented networks} holds. In the remaining part of the paper, it will be convenient to work with the original network, i.e. without the additional artificial vertices and with the orientation of the edges that has been chosen initially.
% All the conditions about Mean field games system remain valid with the new oriented network $\overline{\Gamma}$ and it will be discussed . Therefore, unless otherwise specified, the network $\Gamma$ has the .
\end{rem}

\subsubsection{Function spaces}
\label{sec:function-spaces}
The set of continuous functions on $\Gamma$ is denoted by $C(\Gamma)$
and we set
\[
PC\left(\Gamma\right):=\left\{ v : \Gamma \to \R \;:  \hbox{  for all $\alpha\in\mathcal{A}$}, \left|
    \begin{array}[c]{l}
\hbox{$v_\alpha\in C(0,\ell_\alpha)$ }\\
  \hbox{ $v_\alpha$ can be extended by continuity to $[0,\ell_\alpha]$.}     
    \end{array} \right.
\right\}.
\]
By the definition of piecewise continuous functions $v\in PC(\Gamma)$, for all $\alpha\in \mathcal{A}$, it is possible to extend $v|_{\Gamma_\alpha}$ by continuity   at the endpoints of $\Gamma_\alpha$: if $\Gamma_{\alpha}=[\nu_i,\nu_j]$, we set
\begin{equation}
v|_{\Gamma_{\alpha}}\left(x\right)=\begin{cases}
v_{\alpha}\left(\pi_{\alpha}^{-1}\left(x\right)\right), & \text{if }x\in\Gamma_{\alpha}\backslash\mathcal{V},\\
\displaystyle v_{\alpha}\left(0\right):=\lim_{y\rightarrow0^{+}}v_{\alpha}\left(y\right), & \text{if }x=\nu_{i},\\
\displaystyle v_{\alpha}\left(\ell_{\alpha}\right):=\lim_{y\rightarrow\ell_{\alpha}^{-}}v_{\alpha}\left(y\right), & \text{if }x=\nu_{j}.
\end{cases}\label{eq: v at the vertices}
\end{equation}
For $m\in\mathbb{N}$, the space of $m$-times continuously differentiable functions on $\Gamma$ is defined by
\[
C^{m}\left(\Gamma\right):=\left\{ v\in C\left(\Gamma\right):v_{\alpha}\in C^{m}\left(\left[0,\ell_{\alpha}\right]\right)\text{ for all }\alpha\in\mathcal{A}\right\} ,
\]
and is endowed with the norm
$ 
\left\Vert v\right\Vert _{C^{m}\left(\Gamma\right)}:= {\sum}_{\alpha\in\mathcal{A}}{\sum}_{k\le m}\left\Vert \partial^{k}v_{\alpha}\right\Vert _{L^{\infty}\left(0,\ell_{\alpha}\right)}
$.
For $\sigma\in\left(0,1\right)$, the space $C^{m,\sigma}\left(\Gamma\right)$, 
contains the functions
$v\in C^{m}\left(\Gamma\right)$ such that $\partial^{m}v_{\alpha}\in C^{0,\sigma}\left(\left[0,\ell_{\alpha}\right]\right)$
for all $\alpha\in \mathcal{A}$; it is endowed with the norm
$\displaystyle{
\left\Vert v\right\Vert _{C^{m,\sigma}\left(\Gamma\right)}:=\left\Vert v\right\Vert _{C^{m}\left(\Gamma\right)}+\sup_{\alpha\in\mathcal{A}}\sup_{y\ne z \atop y,z\in\left[0,\ell_{\alpha}\right]}\dfrac{\left|\partial^{m}v_{\alpha}\left(y\right)-\partial^{m}v_{\alpha}\left(z\right)\right|}{\left|y-z\right|^{\sigma}}}$.

For a positive integer $m$ and a function $v\in C^{m}\left(\Gamma\right)$,
we set for $k\le m$,
\begin{equation}
\partial^{k}v\left(x\right)=\partial^{k}v_{\alpha}\left(\pi_{\alpha}^{-1}\left(x\right)\right)\text{ if }x\in\Gamma_{\alpha}\backslash{\mathcal{V}}.\label{eq: derivative}
\end{equation}

Notice that $v\in C^{k}\left(\Gamma\right)$ is continuous on $\Gamma$
but that the derivatives $\partial^{l}v,$ $0<l\le k$ are not defined at the vertices.
For a vertex $\nu$, we define $\partial_{\alpha}v\left(\nu\right)$ as the {\sl outward} directional
derivative of $v|_{\Gamma_{\alpha}}$ at $\nu$ as follows:
\begin{equation}
\partial_{\alpha}v\left(\nu\right):=\begin{cases}
{\displaystyle \lim_{h\rightarrow0^{+}}
\dfrac{v_{\alpha}\left(0\right)-v_{\alpha}\left(h\right)}{h}}, & \text{if }\nu=\pi_{\alpha}\left(0\right),\\
{\displaystyle \lim_{h\rightarrow0^{+}}
\dfrac{v_{\alpha}\left(\ell_{\alpha}\right)-v_{\alpha}\left(\ell_{\alpha}-h\right)}{h}}, & \text{if }\nu=\pi_{\alpha}\left(\ell_{\alpha}\right).
\end{cases}\label{eq: inward derivative}
\end{equation}
For all $i\in I$ and $\alpha\in \cA_i$, setting 
\begin{equation}
  \label{eq:1}
 n_{i\alpha}=\left\{
 \begin{array}[c]{rl}
 1 & \text{if }   \nu_i =    \pi_{\alpha} (\ell_\alpha),\\
 -1 & \text{if }  \nu_i =    \pi_{\alpha} (0),   
 \end{array}\right.
\end{equation}
we have
\begin{equation}
\label{eq:2}
  \partial_\alpha  v(\nu_i)= n_{i\alpha}\,\partial v|_{\Gamma_{\alpha}}(\nu_i)= n_{i\alpha} \, \partial v_ \alpha (\pi^{-1}_\alpha (\nu_i)) .
\end{equation}

% \begin{rem}
% For $\Gamma_{\alpha}=[\nu_i,\nu_j]$ with $i<j$,
% we have $\partial v|_{\Gamma_{\alpha}}(\nu_i)=\partial_\alpha  v(\nu_i)$ and $\partial v|_{\Gamma_{\alpha}}(\nu_j)=-\partial_\alpha v(\nu_j)$,
%  from \eqref{eq: v at the vertices} and \eqref{eq: inward derivative}. 
% \end{rem}

\begin{rem}\label{sec:netw-funct-spac}
Changing the orientation of the edge does not change the value of $\partial_\alpha v(\nu)$ in (\ref{eq: inward derivative}). 
\end{rem}

If for all $\alpha\in\mathcal{A}$, $v_{\alpha}$ is Lebesgue-integrable on $(0,\ell_{\alpha})$,
 then the integral of $v$ on $\Gamma$
is defined by
$
\int_{\Gamma}v\left(x\right)dx=\sum_{\alpha\in\mathcal{A}}\int_{0}^{\ell_{\alpha}}v_{\alpha}\left(y\right)dy$. The space
$
L^{p}\left(\Gamma\right)  =\left\{ v:v|_{\Gamma_{\alpha}}\in L^{p}\left(\Gamma_{\alpha}\right)\text{ for all \ensuremath{\alpha\in\mathcal{A}}}\right\}$,  $p\in [1,\infty]$, 
is endowed with the norm $\left\Vert v\right\Vert _{L^{p}\left(\Gamma\right)}:=  \left( \sum_{\alpha\in\mathcal{A}}\left\Vert v_{\alpha}\right\Vert _{L^{p} \left(0,\ell_{\alpha}\right)}^p \right)^{\frac 1 p}$
if $1\le p<\infty$, and $\max_{\alpha\in \cA} \|v_\alpha\|_{L^\infty\left(0,\ell_{\alpha}\right)}$ if $p=+\infty$.
%=\left( \sum_{\alpha\in\mathcal{A}}\left\Vert v|_{\Gamma_{\alpha}}\right\Vert ^p _{L^{p}\left(\Gamma_{\alpha}\right)}\right)^{\frac 1 p}$.
We shall also need to deal with functions on $\Gamma$ whose restrictions to the edges are  weakly-differentiable:
we shall use the same notations for the weak derivatives. Let us introduce  Sobolev spaces on $\Gamma$:
\begin{defn} \label{def: Sobolev space}
For any integer $s\ge 1$ and any real number $p\ge 1$, the Sobolev space $W^{s,p}(\Gamma)$ is defined as follows:
$
W^{s,p}(\Gamma):=\left\{ v\in C\left(\Gamma\right):v_{\alpha}\in W^{s,p}\left(0,\ell_{\alpha}\right)\text{ for all }\alpha\in\mathcal{A}\right\}$,  
and endowed with the norm
$
\left\Vert v\right\Vert _{W^{s,p}\left(\Gamma\right)}=\left(\sum^{s}_{k=1}\sum_{\alpha\in\mathcal{A}}
\left\Vert \partial^{k}v_{\alpha}\right\Vert _{L^{p}\left(0,\ell_{\alpha}\right)}^{p}+
\left\Vert v\right\Vert _{L^p(\Gamma)}^{p}\right)^{\frac 1 p}$.
We also set $H^s(\Gamma)= W^{s,2}(\Gamma)$.
\end{defn}

\subsection{A class of stochastic processes on $\Gamma$}
\label{sec:class-stoch-proc}
After rescaling the edges, it may be assumed that $\ell_{\alpha}=1$ for all $\alpha\in\mathcal{A}$. 
Let $ \mu_{\alpha} , \alpha\in\mathcal{A}$ and  $ p_{i\alpha}, i\in I,\alpha\in\mathcal{A}_{i}$
be positive constants such that $\sum_{\alpha\in\mathcal{A}_{i}}p_{i\alpha}=1$. Consider also a real valued function $a\in PC(\Gamma)$.% such that, for all $\alpha\in \cA$, $a|_{\Gamma_\alpha} \in C^1(\Gamma_\alpha)$.

As in Remark~\ref{rem:new network}, we make the assumption (\ref{oriented networks}) by possibly adding artificial nodes: if $\nu_i$ is such an artificial node, then $\sharp(\cA_i)=2$, and we assume that $p_{i\alpha}=1/2$ for $\alpha\in \cA_i$. The diffusion parameter $\mu$ has the same value on the two sides of an artificial vertex. Similarly, the function $a$ 
%and its derivative  
does not have jumps across an artificial vertex.

 Let us consider the linear  differential operator:
\begin{equation}
  \label{eq:3}
\mathcal{L}u\left(x\right)=\mathcal{L}_{\alpha}u\left(x\right):=
\mu_{\alpha}\partial^{2}u\left(x\right)+a|_{\Gamma_{\alpha}}\left(x\right)\partial u\left(x\right), \quad 
\hbox{if }x\in \Gamma_\alpha,
\end{equation}
with domain
\begin{equation}
D\left(\mathcal{L}\right):=\left\{ u\in C^{2}\left(\Gamma\right):\sum_{\alpha\in\mathcal{A}_{i}}p_{i\alpha}\partial_{\alpha}u\left(\nu_{i}\right)=0,\; \hbox{ for all }i\in I \right\}.\label{Kirchhoff condition}
\end{equation}
\begin{rem}\label{sec:class-stoch-proc-2}
  Note that  in the definition of $D\left(\mathcal{L}\right)$,  the condition at boundary vertices boils down to a Neumann condition.
\end{rem}
% \begin{rem}
%   \label{sec:class-stoch-proc-3}
%  For a function $v\in D(\cL)$,  the directional derivative of $v$ does not have jump at artificial nodes.
% \end{rem}

Freidlin and Sheu proved in \cite{FS2000} that 
\begin{enumerate}
\item The operator $\cL$ is the infinitesimal generator of a Feller-Markov process on $\Gamma$ with continuous sample paths. The operators $\cL_\alpha$ and the transmission conditions at the vertices
\begin{equation}
  \label{eq:4}
  \sum_{\alpha \in \cA_{i} } p_{i\alpha} \partial_\alpha u(\nu_i)=0
\end{equation}
define such a process in a unique way, see also \cite[Theorem 3.1]{FW1993}. The process can be written $(X_t, \alpha_t)$ where  $X_t\in \Gamma_{\alpha_t}$.
If $X_t=\nu_i$, $i\in I$, $\alpha_t$ is arbitrarily chosen as the smallest index in $\cA_i$. Setting $x_t= \pi_{\alpha_t}(X_t)$
 defines the process $x_t$ with values in $[0,1]$. 
\item There exist
  \begin{enumerate}
  \item  a one dimensional Wiener process $W_t$,
  \item  continuous non-decreasing processes $\ell_{i,t}$, $i\in I$,  which are measurable with respect to the $\sigma$-field generated by $(X_t,\alpha_t)$,
  \item   continuous non-increasing processes $h_{i,t}$, $i\in I$, which are measurable with respect to the $\sigma$-field generated by $(X_t,\alpha_t)$,
  \end{enumerate}
 such that
 \begin{eqnarray}
    \label{eq:5}
  dx_t= \mu_{\alpha_t} dW_t + a_{\alpha_t}(x_t) dt  + d\ell_{i,t} +d h_{i,t},\\
\hbox{$\ell_{i,t}$ increases only when $X_t=\nu_i$ and $x_t=0$,} \nonumber \\
\hbox{$h_{i,t}$ decreases only when $X_t=\nu_i$ and $x_t=1$.} \nonumber
 \end{eqnarray}

\item The following Ito formula holds:  for any real valued function $u \in C^2(\Gamma)$:
\end{enumerate}
\begin{equation}
  \label{eq:6}
  \begin{split}
  u(X_t)=& u(X_0)
\\  &+ \sum_{\alpha\in \mathcal{A}} \int_0^t    \mathds{1}_{\{X_s \in \Gamma_\alpha\backslash \mathcal {V}\}} 
 \left(  \mu_\alpha \partial^2 u(X_s) + a(X_s)\partial u(X_s)  ds   
+ \sqrt{2 \mu_\alpha} \partial u(X_s) dW_s \right)\\ 
& + \sum_{i\in I}  \sum_{\alpha \in \cA_{i}} p_{i\alpha} \partial_\alpha u(\nu_i) (\ell_{i,t}+h_{i,t}) .
  \end{split}
\end{equation}

\begin{rem}
  \label{sec:class-stoch-proc-1}
The assumption that all  the edges have unit length is not restrictive, because we can always rescale the constants $\mu_\alpha$ and the piecewise continuous function $a$. The Ito formula in (\ref{eq:6}) holds when this assumption is not satisfied.
\end{rem}
Consider the invariant measure associated with the process $X_t$. We may assume that it is absolutely continuous 
 with respect to the Lebesgue measure on $\Gamma$. Let $m$ be its density:  
 \begin{equation}
   \label{eq:7}
\mathbb{E}\left[u\left(X_t\right)\right]:=\int_{\Gamma}u\left(x\right)m\left(x\right)dx, \quad  \hbox{ for all } u\in PC(\Gamma).
 \end{equation}
We focus on functions $u\in D\left(\mathcal{L}\right)$. Taking the time-derivative of each member of (\ref{eq:7}),  Ito's formula (\ref{eq:6}) and  (\ref{eq:4}) lead to $
\mathbb{E}\left[  \mathds{1}_{\{X_t\notin \cV\}}   \left(a\partial u (X_t)+\mu\partial^{2}u(X_t)\right) \right] =0$.  
This implies that 
\begin{equation}
  \label{eq:8}
  \int_{\Gamma}\left(a(x)\partial u(x)+\mu\partial^{2}u(x)\right)m(x)dx=0.
\end{equation}
Since for $\alpha \in \cA$, any smooth function on $\Gamma$ compactly supported in $\Gamma_\alpha \backslash \cV$ clearly
belongs to $D({\mathcal L})$, (\ref{eq:8}) implies that $m$ satisfies 
   \begin{equation}
     \label{eq:9}
-\mu_{\alpha}\partial^{2}m +\partial\left(m a\right)=0
   \end{equation}
in the sense of distributions in the edges $\Gamma_\alpha \backslash \cV$, $\alpha\in \cA$. This implies that 
there exists a real number $c_\alpha$ such that  
\begin{equation}
  \label{eq:10}
-\mu_{\alpha}\partial m|_{\Gamma_\alpha} =- m|_{\Gamma_\alpha} a|_{\Gamma_\alpha}+c_\alpha.
\end{equation}
So $m|_{\Gamma_\alpha}$ is $C^1$ regular, and (\ref{eq:10}) is true pointwise. 
Using this information and recalling (\ref{eq:8}), we find that, for all $u\in D(\cL)$,
\begin{displaymath}
\sum_{i\in I} \sum_{\alpha\in\mathcal{A}_{i}}   \mu_{\alpha}m|_{\Gamma_{\alpha}}\left(\nu_{i}\right)\partial_{\alpha}u\left(\nu_i\right)  + \sum_{\beta\in\cA} \int_{\Gamma_\beta} \partial u|_{\Gamma_\beta}(x) \Bigl( -\mu_\beta \partial m|_{\Gamma_\beta}(x) +a|_{\Gamma_\beta}(x)m|_{\Gamma_\beta}(x) \Bigr) dx=0.
\end{displaymath}
% where, for $i\in I$ and $\alpha\in \mathcal{A}_{i}$,
% \[
% d_{i,\alpha}=\begin{cases}
% -1 & \text{if }\pi_{\alpha}\left(\nu_{i}\right)=\ell_{\alpha},\\
% 1 & \text{if }\pi_{\alpha}\left(\nu_{i}\right)=0.
% \end{cases}
% \]
This and (\ref{eq:10}) imply that 
\begin{equation}
  \label{eq:11}
\sum_{i\in I}\sum_{\alpha\in\mathcal{A}_{i}} \mu_{\alpha}m|_{\Gamma_{\alpha}}\left(\nu_{i}\right)\partial_{\alpha}u\left(\nu_i\right) + \sum_{\beta\in\cA} c_\beta \int_{\Gamma_\beta} \partial u|_{\Gamma_\beta}(x)dx=0.
\end{equation}

For all $i\in I$, it is possible to choose a function $u\in D(\cL)$ such that
\begin{enumerate}
\item  $u(\nu_j)=\delta_{i,j}$ for all $j\in I$;
\item $\partial_\alpha  u(\nu_j)=0$ for all $j\in I$ and $\alpha\in \cA_j$. 
\end{enumerate}
Using such a test-function in (\ref{eq:11}) implies that for all $i\in I$,
\begin{equation}
\label{eq:12}
0=
\sum_{\beta\in \mathcal{A}} c_\beta
    \int_{\Gamma_\beta} \partial u|_{ \Gamma_\beta}(x)dx =
 \sum_{j\in I}\sum_{\alpha\in \mathcal{A}_j}
c_\alpha n_{j\alpha}u|_{ \Gamma_\alpha}(\nu_j)
=
 \sum_{\alpha\in \cA_i} n_{i\alpha} c_\alpha,
\end{equation}
where $n_{i\alpha}$ is defined in (\ref{eq:1}).

For all $i\in I$ and $\alpha,\beta\in \cA_i$, it is possible to choose a function $u\in D(\cL)$ such that 
\begin{enumerate}
\item $u$ takes the same value at each vertex of $\Gamma$, thus $\int_{\Gamma_\delta} \partial u|_{\Gamma_\delta}(x)dx=0$ for all $\delta\in \cA$;
\item $\partial_\alpha  u(\nu_i)= 1/p_{i\alpha}$,
 $\partial_\beta  u(\nu_i)= -1/p_{i\beta}$ 
and all the other first order directional derivatives of $u$ at the vertices are $0$.
\end{enumerate}
Using such a test-function  in (\ref{eq:11}) yields
\[
\dfrac{m|_{\Gamma_{\alpha}}\left(\nu_{i}\right)}{\gamma_{i\alpha}}=\dfrac{m|_{\Gamma_{\beta}}\left(\nu_{i}\right)}{\gamma_{i\beta}},\quad\text{for all }\alpha,\beta\in\mathcal{A}_{i},\nu_{i}\in\mathcal{V},
\] 
in which
\begin{equation}
  \label{eq:13} \gamma_{i\alpha}= \frac {p_{i\alpha}}{\mu_\alpha},  \quad \hbox{for all } i\in I,   \alpha\in \cA_i. 
\end{equation}

Next, for $i\in I$,  multiplying (\ref{eq:10}) at $x=\nu_i$ by $n_{i\alpha}$ for all  $\alpha\in \cA_i$,
then  summing over all $\alpha\in \cA_i$, we get $
  \sum_{\alpha\in \cA_i}  \mu_{\alpha}\partial_\alpha m \left(\nu_{i}\right) - n_{i\alpha}\Bigl( m|_{\Gamma_\alpha}\left(\nu_{i}\right) a|_{\Gamma_\alpha}\left(\nu_{i}\right)-c_\alpha\Bigr) =0$,
and using (\ref{eq:12}), we obtain that
\begin{equation}
  \label{eq:14}
\sum_{\alpha\in\mathcal{A}_{i}}
 \mu_{\alpha}\partial_{\alpha}m\left(\nu_{i}\right)- n_{i\alpha}
a|_{\Gamma_{\alpha}}\left(\nu_{i}\right)m|_{\Gamma_{\alpha}}\left(\nu_{i}\right)=0,\quad\text{for all }i\in I.
\end{equation}

Summarizing, we get the following boundary value problem  for $m$ (recalling that the coefficients $n_{i\alpha}$ are defined in (\ref{eq:1})):
\begin{equation}
  \label{eq:15}
\left\{  \begin{split}
-\mu_{\alpha}\partial^{2}m+\partial\left(m a\right)=0,\quad \quad  & x\in\left(\Gamma_{\alpha}\backslash\mathcal{V}\right),\,\alpha\in\mathcal{A},\\
\sum_{\alpha\in\mathcal{A}_{i}}\mu_{\alpha}\partial_{\alpha}m\left(\nu_{i}\right)
-  n_{i\alpha}  a|_{\Gamma_{\alpha}} (\nu_i) m|_{\Gamma_{\alpha}}\left(\nu_{i}\right)=0,\quad \quad  & \nu_{i}\in\mathcal{V},\\
\dfrac{m|_{\Gamma_{\alpha}}\left(\nu_{i}\right)}{\gamma_{i\alpha}}=\dfrac{m|_{\Gamma_{\beta}}\left(\nu_{i}\right)}{\gamma_{i\beta}},\quad \quad  & \alpha,\beta\in\mathcal{A}_{i},\nu_{i}\in\mathcal{V}.    
  \end{split}\right.
\end{equation}

\subsection{Formal derivation of the MFG system on $\Gamma$}\label{subsec: A derivation of the MFG system}
 Consider a continuum of indistinguishable agents moving on the network $\Gamma$. The set of  Borel probability
measures on $\Gamma$ is denoted by  $\mathcal{P}\left(\Gamma\right)$. %It can endowed with the Kantorovitch-Rubinstein distance.
Under suitable assumptions,  the theory of MFGs asserts that the 
distribution of states is absolutely continuous with respect to  Lebesgue measure on $\Gamma$. Hereafter,  $m$ stands for the density of the distribution of states: $m\ge 0$ and $\int_{\Gamma}m(x)dx=1$.

The  state of a representative agent at time $t$ is a time-continuous controlled stochastic process $X_t$ in $\Gamma$,
 as defined in Section~\ref{sec:class-stoch-proc}, where the control is the drift $a_t$, supposed to be of the form $a_t=a(X_t)$.
The function $X\mapsto a(X)$ is the feedback. 

For a representative agent, the optimal control problem  is of the form:
\begin{equation}
\rho:=\inf_{a_s}\liminf_{T\rightarrow+\infty}\dfrac{1}{T}{\mathbb{E}_{x}
\left[\int_{0}^{T}L\left(X_s,a_{s}\right)+V\left[m\left(\cdot,s\right) (X_s)\right]ds \right]},\label{ergodic constant}
\end{equation}
where $\mathbb{E}_{x}$ stands for the expectation conditioned by the event $X_0=x$. The functions and operators involved in (\ref{ergodic constant})
will be described below.

Let us assume that there is an optimal feedback law, i.e. a function $a^\star$ defined on $\Gamma$ which is sufficiently regular in the edges of the network, such that the optimal control at time $t$ is given by $a_t ^\star=a^\star(X_t)$. Then, almost surely if $X_t\in \Gamma_\alpha\backslash \cV$,
$  d  \pi_{\alpha}^{-1} (X_t)=  a^\star_{\alpha}(\pi_{\alpha}^{-1} (X_t)) dt + \sqrt{ 2\mu_{\alpha}} dW_t$.  An informal way to describe the behavior of the process at the vertices is  as follows: if $X_t$ hits $\nu_{i}\in\mathcal{V}$, then it enters $\Gamma_{\alpha}$, $\alpha\in\mathcal{A}_{i}$  with probability $p_{i\alpha}>0$.

Let us discuss the ingredients in (\ref{ergodic constant}): the running cost depends separately on the control and on the  distribution of states.
 The contribution of the  distribution of states involves the coupling cost operator, which can either be nonlocal, i.e.
  $V:\mathcal{P}\left(\Gamma\right)\rightarrow   \mathcal{C}^2 (\Gamma)$,
or local, i.e. $V[m](x)= F(m(x))$   assuming that $m$ is absolutely continuous with respect to the Lebesgue measure, where $F: \R^+ \to \R$ is a continuous function.

The contribution of the control involves the  Lagrangian $L$, i.e. a real valued function defined on 
$ \left(\cup_{\alpha \in \mathcal{A}} \Gamma_\alpha \backslash \mathcal{V} \right) \times \R$. 
 If $x\in \Gamma_\alpha \backslash \mathcal{V}$ and $a\in \R$, $L(x,a)=L_\alpha(\pi_\alpha^{-1}(x),a)$, where $L_\alpha$ is a continuous real valued function defined on $[0,\ell_\alpha]\times \R$. We assume that $\lim_{|a|\to \infty} \inf_{y\in \Gamma_\alpha}  {L_\alpha(y,a)}/{|a|}=+\infty$. 
Further assumptions on $L$ and $V$ will be made below.

Under suitable assumptions,  the Ito calculus recalled in Section~\ref{sec:class-stoch-proc}
 and the dynamic programming principle lead to the following ergodic 
Hamilton-Jacobi equation on $\Gamma$,  more precisely the following boundary value problem:
\begin{equation}
\begin{cases}
-\mu_{\alpha}\partial^{2}v+H\left(x,\partial v\right)+\rho=V\left[m\right](x), & x\in\left(\Gamma_{\alpha}\backslash\mathcal{V}\right),\alpha\in\mathcal{A},\\
{\displaystyle \sum_{\alpha\in\mathcal{A}_{i}}\gamma_{i\alpha}\mu_{\alpha}\partial_{\alpha}v\left(\nu_{i}\right)=0,} 
& \nu_{i}\in\mathcal{V},\\
v|_{\Gamma_{\alpha}}\left(\nu_{i}\right)=v|_{\Gamma_{\beta}}\left(\nu_{i}\right), & \alpha,\beta\in\mathcal{A}_{i},\nu_{i}\in\mathcal{V},\\ \ds \int_{\Gamma}v(x)dx=0.
\end{cases}\label{eq: introduction HJ}
\end{equation}
We refer to  \cite{LL2006-A,LL2007} for the interpretation of the value function $v$ and the ergodic cost $\rho$.

Let us comment  the different equations in (\ref{eq: introduction HJ}):
\begin{enumerate}
\item The Hamiltonian $H$ is a  real valued function defined
on $ \left(\cup_{\alpha \in \mathcal{A}} \Gamma_\alpha \backslash \mathcal{V} \right) \times \R$. 
For $x\in \Gamma_\alpha \backslash \mathcal{V} $ and $p\in \R$,
\[
H\left(x,p\right)=\sup_{a}\left\{ -a p-L_{\alpha}\left(\pi_\alpha^{-1}(x),a\right)\right\},
\]
The Hamiltonian is supposed to be $ C^1$  and coercive with respect to $p$ uniformly in $x$. 
\item  The second equation in \eqref{eq: introduction HJ} is  a Kirchhoff
transmission condition (or Neumann boundary condition if $\nu_i\in\partial\Gamma$);  it is the consequence of the assumption
on the behavior of $X_s$ at vertices. It involves the positive constants  $\gamma_{i\alpha}$ defined in (\ref{eq:13}).
\item
 The third condition means in particular that $v$ is continuous at the vertices.
\item 
The fourth equation is a normalization condition.
\end{enumerate}
If \eqref{ergodic constant} has a smooth solution, then it provides a feedback law for the optimal control problem, i.e.
 \[
 a^{\star}(x)=-\partial_{p}H\left(x,\partial v\left(x \right)\right).
 \]
At the MFG equilibrium, $m$ is the density of the invariant measure associated with the optimal feedback law, so, according to Section~\ref{sec:class-stoch-proc}, it satisfies (\ref{eq:15}), 
where $a$ is replaced by $a^\star=-\partial_{p}H\left(x,\partial v\left(x\right)\right)$. We end up with the following system:
\begin{equation}
{\displaystyle \begin{cases}
-\mu_{\alpha}\partial^{2}v+H\left(x,\partial v\right)+\rho=V\left([m]\right), & x\in\Gamma_{\alpha}\backslash\mathcal{V},\alpha\in \mathcal{A},\\
\mu_{\alpha}\partial^{2}m+\partial\left(m\partial_{p}H\left(x,\partial v\right)\right)=0, & x\in\Gamma_{\alpha}\backslash\mathcal{V},\alpha\in \mathcal{A},\\
{\displaystyle \sum_{\alpha\in\mathcal{A}_{i}}\gamma_{i\alpha}\mu_{\alpha}\partial_{\alpha}\left(\nu_{i}\right)=0,} & \nu_{i}\in\mathcal{V},\\
{\displaystyle \sum_{\alpha\in\mathcal{A}_{i}}\left[\mu_{\alpha}\partial_{\alpha}m  \left(\nu_{i}\right)+
n_{i\alpha} \partial_{p}H_{\alpha}\Bigl(\nu_{i}, \partial v|_{\Gamma_\alpha} (\nu_i)\Bigr)m|_{\Gamma_{\alpha}}\left(\nu_{i}
\right)\right]=0,} & \nu_{i}\in\mathcal{V},\\
v|_{\Gamma_{\alpha}}\left(\nu_{i}\right)=v|_{\Gamma_{\beta}}\left(\nu_{i}\right), & \alpha,\beta\in\mathcal{A}_{i},\nu_{i}\in\mathcal{V},
\\\dfrac{m|_{\Gamma_{\alpha}}\left(\nu_{i}\right)}{\gamma_{i\alpha}}=\dfrac{m|_{\Gamma_{\beta}}\left(\nu_{i}\right)}{\gamma_{i\beta}}, & \alpha,\beta\in\mathcal{A}_{i},\nu_{i}\in\mathcal{V},\\
\ds \int_{\Gamma}v\left(x\right)dx=0,\quad\int_{\Gamma}m\left(x\right)dx=1,\quad m\ge0.
\end{cases}}\label{eq: MFG system}
\end{equation}
At  a vertex $\nu_i$, $i\in I$, the transmission conditions for both  $v$ and $m$
 consist of $d_{\nu_{i}}= \sharp( \cA_i)$ linear relations, which is the appropriate number of relations to 
have a well posed problem. If $\nu_i\in \partial \Gamma$, there is of course only one Neumann like  condition for $v$ and for $m$.
%    $ linear conditions for
% each functions at $\nu_{i}\in\mathcal{V}$, 
% and hence they univocally
% determine the value $v_{\Gamma_{\alpha}}\left(\nu_{i},t\right)$
% and $m|_{\Gamma_{\alpha}}\left(\nu_{i},t\right),$ with $\alpha\in\mathcal{A}_{i}$.

\begin{rem}
In \cite{CM2016}, the authors assume  that   $ \gamma_{i\alpha}=\gamma_{i\beta}$ for all $i\in I$, $\alpha,\beta\in\mathcal{A}_i$. Therefore, the density  $m$ does not have jumps across the transition vertices.
\end{rem}

\subsection{Assumptions and main results}

\subsubsection{Assumptions}

\label{sec:assumptions}
Let $(\mu_\alpha)_{\alpha\in \mathcal{A}}$ be a family of positive numbers, and 
for each $i\in I$ let $( \gamma_{i\alpha})_{\alpha\in\mathcal{A}_{i}} $ be a family of positive numbers such that 
$\sum_{\alpha\in \cA_i} \gamma_{i\alpha}\mu_\alpha=1$.
\\
 Consider the Hamiltonian $H:\Gamma\times\mathbb{R}\rightarrow\mathbb{R}$,
with $H|_{\Gamma_{\alpha}}:\Gamma_{\alpha}\times\mathbb{R}\rightarrow\mathbb{R}$.
We assume that, for some positive constants  $C_{0},C_{1},C_{2}$ %$,C_{3}$
and $q\in\left(1,2\right]$, 
\begin{align}
 & H_{\alpha}\in C^{1}\left(\left[0,\ell_{\alpha}\right]\times\mathbb{R}\right);\label{eq: Hamiltonian}\\
 & H_{\alpha}\left(x,\cdot\right)\text{is convex in }p\text{ for each }x\in\left[0,\ell_{\alpha}\right];\label{eq: Hamiltonian is convex}\\
 & H_{\alpha}\left(x,p\right)\ge C_{0}\left|p\right|^{q}-C_{1}\text{ for }\left(x,p\right)\in\left[0,\ell_{\alpha}\right]\times\mathbb{R};\label{eq: super quadratic}\\
 & \left|\partial_{p}H_{\alpha}\left(x,p\right)\right|\le C_{2}\left(\left|p\right|^{q-1}+1\right)\text{ for }\left(x,p\right)\in\left[0,\ell_{\alpha}\right]\times\mathbb{R}.\label{eq: derivative of H is bounded}
%\\
%& \partial_{p}H_{\alpha}\left(x,p\right)\cdot p-H_{\alpha}\left(x,p\right)\ge\delta H_{\alpha}\left(x,p\right)-C_{3}\text{ for }
%\left(x,p\right)\in\left[0,\ell_{\alpha}\right]\times\mathbb{R}.\label{eq: convexity Hamiltonian}
\end{align}
\begin{rem}
From \eqref{eq: derivative of H is bounded}, there
exists a positive constant $C_{q}$ such that
\begin{equation}
\left|H_{\alpha}\left(x,p\right)\right|\le C_{q}\left(\left|p\right|^{q}+1\right),\quad\text{for all }\left(x,p\right)\in\left[0,\ell_{\alpha}\right]\times\mathbb{R}.\label{eq: sup quadractic}
\end{equation}
\label{rem: derivative of H is bounded}
\end{rem}
Below, we shall focus on local coupling operators $V$, namely
\begin{equation}
V\left[\tilde{m}\right]\left(x\right)=F\left(m\left(x\right)\right)\text{ with }
F\in C\left(\left[0,+\infty\right); \R \right),\label{eq: coupling F}
\end{equation}
for all $\tilde{m}$ which are absolutely continuous with respect
to the Lebesgue measure and such that $d\tilde{m}\left(x\right)=m\left(x\right)dx$.
We shall also suppose that   $F$ is bounded from below, i.e.,
there exists a positive constant $M$ such that
\begin{equation}
F\left(r\right)\ge-M,\quad\text{for all }r\in\left[0,+\infty\right).\label{ineq: F is bounded from below}
\end{equation}

\subsubsection{Function spaces related to the Kirchhoff conditions}
\label{sec:funct-spac-relat}
Let us introduce two function spaces on $\Gamma$,  which will be the key ingredients in order to build  weak
solutions of \eqref{eq: MFG system}.
\begin{defn}
\label{def: functional spaces}
We define two Sobolev spaces, $V:=H^1(\Gamma)$,
% \begin{align}
% V: & = H^1(\Gamma), %\left\{ v\in C\left(\Gamma\right):v{}_{\alpha}\in H^{1}\left(0,\ell_{\alpha}\right)\text{ for }\alpha\in\mathcal{A}\right\}
%  \label{eq: space V}
% \end{align}
see Definition \ref{def: Sobolev space}, 
% which is a Hilbert space,
% endowed with the norm $\left\Vert v\right\Vert _{V}=\left (\sum_{\alpha\in\mathcal{A}}\left\Vert v_{\alpha}\right\Vert^2 _{H^{1}\left(0,\ell_{\alpha}\right)}\right) ^{\frac 1 2}$
and 
\begin{equation}
W:=\left\{
  \begin{array}[c]{ll}
w:\Gamma\rightarrow\mathbb{R}: & \;w_{\alpha}\in H^{1}\left(0,\ell_{\alpha}\right) 
\text{ for all }\alpha \in \cA, \\ & \dfrac{w|_{\Gamma_{\alpha}}\left(\nu_{i}\right)}{\gamma_{i\alpha}}=\dfrac{w|_{\Gamma_{\beta}}\left(\nu_{i}\right)}{\gamma_{i\beta}}\text{ for all }
i\in I, \; \alpha,\beta\in\mathcal{A}_{i}    
  \end{array}
\right\} \label{eq: space W}
\end{equation}
which is also a Hilbert space,
endowed with the norm $\left\Vert w\right\Vert _{W}=\left(\sum_{\alpha\in\mathcal{A}}\left\Vert w_{\alpha}\right\Vert^2 _{H^{1}\left(0,\ell_{\alpha}\right)}\right)^{\frac 1 2}$. 
\end{defn}

% \begin{rem}
% In \cite{CM2016}, the authors  assume that for all $i\in I$,
% $\gamma_{i\alpha}=\gamma_{i\beta}$ for all $\alpha,\beta\in\mathcal{A}_{i}$, leading to the Kirchhoff condition:  
% $\sum_{\alpha\in\mathcal{A}_{i}}\mu_{\alpha}\partial_{\alpha}v\left(\nu_{i}\right)=0$.
% In this case, the two function spaces $V$ and $W$ coincide.
% \end{rem}

\begin{defn}
\label{def: test function}
Let the functions $\psi\in W$ and $\phi\in PC(\Gamma)$ be defined as follows:
\begin{equation}
\begin{cases}
\psi_{\alpha}\text{ is affine on }\left(0,\ell_{\alpha}\right),\\
\psi|_{\Gamma_{\alpha}}\left(\nu_{i}\right)=\gamma_{i\alpha},\text{ if }\alpha\in\mathcal{A}_{i},\\
\psi\text{ is constant on the edges \ensuremath{\Gamma_{\alpha}} which touch the boundary of \ensuremath{\Gamma}}.
\end{cases}\label{eq: test function 1}
\end{equation}
\begin{equation}
\begin{cases}
\phi_{\alpha}\text{ is affine on }\left(0,\ell_{\alpha}\right),\\
\phi|_{\Gamma_{\alpha}}\left(\nu_{i}\right)=\dfrac{1}{\gamma_{i\alpha}},\text{ if }\alpha\in\mathcal{A}_{i},\\
\phi\text{ is constant on the edges \ensuremath{\Gamma_{\alpha}} which touch the boundary of \ensuremath{\Gamma}}.
\end{cases}\label{eq: test function 2}
\end{equation}
Note that both functions $\psi,\phi$ are positive and bounded. 
We  set $\overline{\psi}=\max_{\Gamma}\psi$, $\underline{\psi}=\min_{\Gamma}\psi$, $\overline{\phi}=\max_{\Gamma}\phi$, $\underline{\phi}=\min_{\Gamma}\phi$.
\end{defn}

\begin{rem}
\label{rem: test function}One can see that $v\in V\longmapsto v\psi$
is an isomorphism from $V$ onto $W$ and $w\in W\longmapsto w\phi$
is the inverse isomorphism.
\end{rem}

\begin{defn}\label{sec:funct-spac-relat-1}
Let the function space $\mathcal{W}\subset W$  be defined as follows:
\begin{equation}
\mathcal{W}:=\left\{
  \begin{array}[c]{ll}
m:\Gamma\to \R: & m_{\alpha}\in C^{1}\left(\left[0,\ell_{\alpha}\right]\right)\text{ for all }\alpha \in \mathcal{A}, \\
&\dfrac{m|_{\Gamma_{\alpha}}\left(\nu_{i}\right)}{\gamma_{i\alpha}}=\dfrac{m|_{\Gamma_{\beta}}\left(\nu_{i}\right)}{\gamma_{i\beta}}\text{ for all \ensuremath{i\in I,\alpha,\beta\in\mathcal{A}_{i}}}    
  \end{array}
\right\} .\label{eq: regular space}
\end{equation}
\end{defn}

\subsubsection{Main result}
\label{sec:main-result}
\begin{defn}\label{sec:main-result-1}
A solution of the Mean Field Games system \eqref{eq: MFG system}
is a triple $\left(v,\rho,m\right)\in C^{2}\left(\Gamma\right)\times\mathbb{R}\times\mathcal{W}$
such that $\left(v,\rho\right)$ is a classical solution of
\begin{equation}
  \label{eq:16}
\begin{cases}
-\mu_{\alpha}\partial^{2}v+H\left(x,\partial v\right)+\rho=F\left(m\right), & \text{in }\Gamma_{\alpha}\backslash\mathcal{V},\alpha\in\mathcal{A},\\
{\displaystyle \sum_{\alpha\in\mathcal{A}_{i}}\gamma_{i\alpha}\mu_{\alpha}\partial_{\alpha}v\left(\nu_{i}\right)}=0, & \text{if }\nu_{i}\in\mathcal{V},
\end{cases}
\end{equation}
(note that $v$ is continuous at the vertices from the definition of $C^2 (\Gamma)$),
and $m$ satisfies
\begin{equation}
  \label{eq:17}
\sum_{\alpha\in\mathcal{A}}\int_{\Gamma_{\alpha}}\left[\mu_{\alpha}\partial m\partial u+\partial\left(m\partial_{p}H\left(x,\partial v\right)\right)u\right]dx=0,\quad\text{for all }u\in V,
\end{equation}
where $V$ is given in Definition \ref{def: functional spaces}.
\end{defn}

We are ready to state the main result:
\begin{thm}
\label{thm: MFG system}  If assumptions \eqref{eq: Hamiltonian}-\eqref{eq: derivative of H is bounded} and 
\eqref{eq: coupling F}-\eqref{ineq: F is bounded from below} are satisfied,
%  and if furthermore $F$ is bounded from below, i.e.,
% there exists a positive constant $M$ such that
% \begin{equation}
% F\left(r\right)\ge-M,\quad\text{for all }r\in\left[0,+\infty\right),\label{ineq: F is bounded from below}
% \end{equation}
 then
there exists a solution $\left(v,m,\rho\right)\in C^{2}\left(\Gamma\right)\times\mathcal{W}\times\mathbb{R}$
of \eqref{eq: MFG system}. If $F$ is locally Lipschitz continuous, then $v \in C^{2,1}(\Gamma)$. Moreover if $F$ is strictly increasing, then the solution is unique. 
\end{thm}

\begin{rem}\label{sec:main-result-2}
  The proof of the existence result in  \cite{CM2016} is valid only in the case when the coupling cost $F$ is bounded.
\end{rem}

% \begin{rem}
% If $F$ is bounded by some constant $C$, then by
% Lemma \ref{thm: main result with F bounded} below, assumption \eqref{eq: convexity Hamiltonian} is no longer necessary 
% for  the existence result.
% \end{rem}

\begin{rem}
The existence result in Theorem \ref{thm: MFG system}  holds if we assume that the coupling operator $V$ is  non local  and regularizing, i.e., $ V$ is a continuous map 
from $\mathcal{P}$ to a bounded subset of $\mathcal{F}$,
with
$\mathcal{F}:  =\left\{ f:\Gamma\rightarrow\mathbb{R}:\;f|_{\Gamma_{\alpha}}\in C^{0,\sigma}\left(\Gamma_{\alpha}\right)\right\}$.
The proof, omitted in what follows, is similar to that of Lemma \ref{thm: main result with F bounded} below. 
\end{rem}

%\begin{rem}
%For any $v\in V$ and $w\in W$, since $H^{1}\left(0,\ell_{\alpha}\right)$
%is compactly embedded in $C\left(\left[0,\ell_{\alpha}\right]\right)$,
%it is possible to extend $v$ and $w$ at all vertices of $\Gamma$
%and we set
%\[
%v|_{\Gamma_{\alpha}}\left(\nu_{i}\right)=\lim_{y\in\left(0,\ell_{\alpha}\right),\pi_{\alpha}\left(y\right)\rightarrow\nu_{i}}v_{\alpha}\left(y\right),\quad\text{for all \ensuremath{\alpha\in\mathcal{A}} and \ensuremath{\nu_{i}\in\mathcal{V}\cap\Gamma_{\alpha}.}}
%\]
%\end{rem}

\section{\label{sec: linear problems}Preliminary: A class of linear boundary
value problems}

This section contains  elementary results on the solvability of some linear boundary value problems on $\Gamma$.  To the best of our knowledge, these results are not available in the literature.

\subsection{A first class of problems}
We recall that the constants $\mu_\alpha$ and $\gamma_{i\alpha}$ are defined in Section \ref{sec:class-stoch-proc}.
Let $\lambda$ be a positive number. We start with very simple  linear
boundary value problems, in which the only difficulty is the Kirchhoff condition:
\begin{equation}
\begin{cases}
-\mu_{\alpha}\partial^{2}v+\lambda v=f, & \text{in }\Gamma_{\alpha}\backslash\mathcal{V},\alpha\in\mathcal{A},\\
v|_{\Gamma_\alpha} (\nu_i) =v|_{\Gamma_\beta} (\nu_i), \quad & \alpha,\beta \in \cA_i,\; i\in I,
\\
{\displaystyle \sum_{\alpha\in\mathcal{A}_{i}}\gamma_{i\alpha}\mu_{\alpha}\partial_{\alpha}v\left(\nu_{i}\right)=0}, &  i\in I,
\end{cases}\label{eq: linear 3}
\end{equation}
where $f\in W'$, $W'$ is the topological dual  of $W$.
\begin{rem}
We have already noticed that, if $\nu_{i}\in\partial\Gamma$, the last condition
in \eqref{eq: linear 3} boils down to a standard Neumann boundary condition
$\partial_{\alpha}v\left(\nu_{i}\right)=0$, in which $\alpha$ is the unique element of $\cA_i$. 
Otherwise, if $\nu_{i}\in\mathcal{V}\backslash\partial\Gamma$, the last condition
in \eqref{eq: linear 3} is the Kirchhoff condition discussed above.
\end{rem}

\begin{defn}
\label{def: B_lambda}A weak solution of \eqref{eq: linear 3} is
a function $v\in V$ such that
\begin{equation}
\mathscr{B}_{\lambda}\left(v,w\right)=\left\langle f,w\right\rangle _{W',W},\quad\text{for all }w\in W,\label{eq: linear 3 weak form}
\end{equation}
where $\mathscr{B}_{\lambda}:V\times W\rightarrow\mathbb{R}$ is the
bilinear form defined as follows:
\[
\mathscr{B}{}_{\lambda}\left(v,w\right)=\sum_{\alpha\in\mathcal{A}}\int_{\Gamma_{\alpha}}\left(\mu_{\alpha}\partial v\partial w+\lambda vw\right)dx.
\]
\end{defn}

\begin{rem}
Formally, \eqref{eq: linear 3 weak form} is obtained by testing the first line of \eqref{eq: linear 3}
by $w\in W$,  integrating  by part the left hand side on each $\Gamma_\alpha$ and summing over $\alpha\in \cA$. 
There is no  contribution from the vertices,  because of  the Kirchhoff conditions on the one hand and  on the other hand the  jump conditions satisfied by the elements of $W$.
\end{rem}
\begin{rem}\label{sec:first-class-problems}
By using the fact that $\Gamma_\alpha$ are line segments, i.e. one dimensional sets and solving the differential equations,
 we see that if $v$ is a weak solution of $\eqref{eq: linear 3}$ with $f\in PC\left(\Gamma\right)$, then $v\in C^{2}\left(\Gamma\right)$. 
\end{rem}

Let us first study the homogeneous case, i.e. $f=0$.
\begin{lem}
\label{lem: linear 1}The function $v=0$ is the unique solution of the following boundary value problem
\begin{equation}
\begin{cases}
-\partial^{2}v+\lambda v=0, & \text{in }\Gamma_{\alpha}\backslash\mathcal{V},\alpha\in\mathcal{A},\\
v|_{\Gamma_\alpha} (\nu_i) =v|_{\Gamma_\beta} (\nu_i), \quad & \alpha,\beta \in \cA_i,\; i\in I,
\\
{\displaystyle \sum_{\alpha\in\mathcal{A}_{i}}}\gamma_{i\alpha}\mu_{\alpha}\partial_{\alpha}v\left(\nu_{i}\right)=0, & i\in I,
\end{cases}\label{eq: linear 1}
\end{equation}
%has a unique solution: $v=0$.
\end{lem}

\begin{proof}
%[Proof of Lemma \ref{lem: linear 1}]
Let $\mathcal{I}_{i}:=\left\{ k\in I:\;  k\not= i;\; \nu_{k}\in\Gamma_{\alpha}\hbox{ for some }\alpha\in\mathcal{A}_{i}\right\} $
be the set of indices of the vertices which are connected to $\nu_{i}$.
By Remark \ref{rem:new network}, it is not restrictive to assume (in the remainder of the proof) that  for all $k\in\mathcal{I}_{i}$, $\Gamma_{\alpha}=\Gamma_{\alpha_{ik}}=\left[\nu_{i},\nu_{k}\right]$ is oriented from $\nu_{i}$ to $\nu_{k}$.

For $k\in\mathcal{I}_{i}$, $\Gamma_\alpha=[\nu_i, \nu_k]$, using the parametrization \eqref{eq: parametrization},
the linear differential eqaution \eqref{eq: linear 1} in the edge $\Gamma_{\alpha}$
is
\[
-v_{\alpha}''\left(y\right)+\lambda v_{\alpha}\left(y\right)=0,\quad\text{in }\left(0,\ell_{\alpha}\right),
\]
whose solution is
\begin{equation}
v_{\alpha}\left(y\right)=\zeta_{\alpha}\cosh\left(\sqrt{\lambda}y\right)
+\xi_{\alpha}\sinh\left(\sqrt{\lambda}y\right),\label{eq: explicit formula}
\end{equation}
with
\[
\begin{cases}
\zeta_{\alpha}=v_{\alpha}\left(0\right)=v\left(\nu_{i}\right),\\
\zeta_{\alpha}\cosh\left(\sqrt{\lambda}\ell_{\alpha}\right)+
\xi_{\alpha}\sinh\left(\sqrt{\lambda}\ell_{\alpha}\right)=v_{\alpha}\left(\ell_{\alpha}\right)=v\left(\nu_{k}\right).
\end{cases}
\]
It follows that $\partial_{\alpha}v\left(\nu_{i}\right)= -\sqrt{\lambda}\xi_{\alpha}=-\dfrac{\sqrt{\lambda}}{\sinh\left(\sqrt{\lambda}\ell_{\alpha}\right)}\left[v\left(\nu_{k}\right)-v\left(\nu_{i}\right)\cosh\left(\sqrt{\lambda}\ell_{\alpha}\right)\right]$.
Hence, the transmission condition in \eqref{eq: linear 1} becomes: for all $i\in I$,
\begin{align*}
0  =\sum_{\alpha\in\mathcal{A}_{i}}\gamma_{i\alpha}\mu_{\alpha}\partial_{\alpha}v\left(\nu_{i}\right)
%   &=\sum_{k\in\mathcal{I}_{i}}\dfrac{\sqrt{\lambda}\gamma_{i\alpha_{ik}}\mu_{\alpha_{ik}}}{\sinh\left(\sqrt{\lambda}\ell_{\alpha_{ik}}\right)}
% \left[v\left(\nu_{k}\right)-v\left(\nu_{i}\right)\cosh\left(\sqrt{\lambda}\ell_{\alpha_{ik}}\right)\right]\\
 %&
 =\sum_{k\in\mathcal{I}_{i}}\dfrac{\sqrt{\lambda}\gamma_{i\alpha_{ik}}\mu_{\alpha_{ik}}\cosh\left(\sqrt{\lambda}\ell_{\alpha_{ik}}\right)}{\sinh\left(\sqrt{\lambda}\ell_{\alpha_{ik}}\right)}v\left(\nu_{i}\right)-\sum_{k\in\mathcal{I}_{i}}\dfrac{\sqrt{\lambda}\gamma_{i\alpha_{ik}}\mu_{\alpha_{ik}}}{\sinh\left(\sqrt{\lambda}\ell_{\alpha_{ik}}\right)}v\left(\nu_{k}\right).
\end{align*}
 Therefore, we obtain a   system of linear equations of the form $MU=0$ with $M=
\left(M_{ij}\right)_{1\le i,j\le N}$, $N=\sharp(I)$, and $U=\left(v\left(\nu_{1}\right),\ldots,v\left(\nu_{N}\right)\right)^{T}$, where $M$ is defined by
\[
\begin{cases}
M_{ii} & =\ds \sum_{k\in\mathcal{I}_{i}}\gamma_{i\alpha_{ik}}\mu_{\alpha_{ik}}\dfrac{\cosh\left(\sqrt{\lambda}\ell_{\alpha_{ik}}\right)}{\sinh\left(\sqrt{\lambda}\ell_{\alpha_{ik}}\right)}>0,\\
M_{ik} & =\ds \dfrac{-\gamma_{i\alpha_{ik}}\mu_{\alpha_{ik}}}{\sinh\left(\sqrt{\lambda}\ell_{\alpha_{ik}}\right)}\le0,\quad k\in\mathcal{I}_{i},\\
M_{ik} & =0,\quad k\notin\mathcal{I}_{i}.
\end{cases}
\]
For all $i\in I$, since $\cosh\left(\sqrt{\lambda}\ell_{\alpha_{ik}}\right)>1$
for all $k\in\mathcal{I}_{i}$, the sum of the entries on each row
is positive and $M$ is diagonal dominant. Thus, $M$ is invertible and  $U=0$ is the unique solution of the system. Finally, by solving the ODE in each edge
$\Gamma_{\beta}$ with $v_{\beta}\left(0\right)=v_{\beta}\left(\ell_{\beta}\right)=0$,
we get that $v=0$ on $\Gamma$.
\end{proof}
Let us now study the non-homogeneous problems \eqref{eq: linear 3}.
\begin{lem}
\label{lem: linear 3} For any $f$ in $W'$, \eqref{eq: linear 3}
has a unique weak solution $v$ in $V$, see Definition \ref{def: functional spaces}.
Moreover, there exists a constant $C$ such that $\left\Vert v\right\Vert _{V}\le C\left\Vert f\right\Vert _{W'}.$
\end{lem}

\begin{proof}
%[Proof of Lemma \ref{lem: linear 3}]
First of all, we claim that for
$\lambda_{0}>0$ large enough and any $f\in W'$, the problem
\begin{equation}
\mathscr{B}_{\lambda}\left(v,w\right)+\lambda_{0}\left(v,w\right)=\left\langle f,w\right\rangle_{W',W}\label{eq: lambda0}
\end{equation}
has a unique solution $v\in V$. Let us prove the claim. Let $v\in V$,
then $\hat{w}:=v\psi$ belongs to $W$, where $\psi$ is given by
Definition \ref{def: test function}. Let us set $\overline{\partial\psi}:=\max_{\Gamma}\left|\partial\psi\right|$ and 
$\underline{\psi}:=\min_{\Gamma}\psi >0$, 
($\partial\psi$ is bounded, see Definition \ref{def: test function}); we get 
\begin{align}
\mathscr{B}_{\lambda}\left(v,\hat{w}\right)+\lambda_{0}\left(v,\hat{w}\right) %& =\sum_{\alpha\in\mathcal{A}}\int_{\Gamma_{\alpha}}\left(\mu_{\alpha}\partial v\partial\hat{w}+\lambda v\hat{w}+\lambda_{0}v\hat{w}\right)dx\nonumber \\
 & =\sum_{\alpha\in\mathcal{A}}\int_{\Gamma_{\alpha}}\left[\mu_{\alpha}\left|\partial v\right|^{2}\psi+\mu_{\alpha}\left(v\partial v\right)\partial\psi+\left(\lambda+\lambda_{0}\right)v^{2}\psi\right]dx\nonumber \\
% & \ge\sum_{\alpha\in\mathcal{A}}\int_{\Gamma_{\alpha}}\left[\mu_{\alpha}\left|\partial v\right|^{2}\underline{\psi}-\mu_{\alpha}\left|v\right|\left|\partial v\right|\overline{\partial\psi}+\left(\lambda+\lambda_{0}\right)v^{2}\underline{\psi}\right]dx\nonumber \\
 & \ge\sum_{\alpha\in\mathcal{A}}\int_{\Gamma_{\alpha}}\left[\dfrac{\mu_{\alpha}\underline{\psi}}{2}\left|\partial v\right|^{2}+\left(\lambda_{0}\underline{\psi}-\dfrac{\mu_{\alpha}\overline{\partial\psi}^{2}}{2\underline{\psi}}\right)v^{2}\right]dx.\label{ineq: norm V-1}
\end{align}
 When $\lambda_{0}\ge\dfrac{\mu_{\alpha}}{2}+\dfrac{\mu_{\alpha}\overline{\partial\psi}^{2}}{2\underline{\psi}^{2}}$
for all $\alpha\in\mathcal{A}$, we obtain that 
$
\mathscr{B}_{\lambda}\left(v,\hat{w}\right)+\lambda_{0}\left(v,\hat{w}\right)\ge\dfrac{\underline{\mu}~\underline{\psi}}{2}\left\Vert v\right\Vert _{V}^{2}\ge\dfrac{\underline{\mu}~\underline{\psi}}{2C_{\psi}}\left\Vert v\right\Vert _{V}\left\Vert \hat{w}\right\Vert _{W}$,
using the fact that, from Remark \ref{rem: test function}, there exists a positive
constant $C_{\psi}$ such that $\left\Vert v\psi\right\Vert _{W}\le C_{\psi}\left\Vert v\right\Vert _{V}$
for all $v\in V$. This yields
\begin{equation}
\inf_{v\in V}\sup_{w\in W}\dfrac{\mathscr{B}_{\lambda}\left(v,w\right)+\lambda_{0}\left(v,w\right)}{\left\Vert v\right\Vert _{V}\left\Vert w\right\Vert _{W}}\ge\dfrac{\underline{\mu}~\underline{\psi}}{2C_{\psi}}.\label{ineq: norm W}
\end{equation}
Using a similar argument for any $w\in W$ and $\hat{v}=w\phi$, where $\phi$ is given
 in Definition \ref{def: test function}, we obtain that  for $\lambda_{0}$ large enough,
 there exist a positive constant $C_{\phi}$ such that
\begin{equation}
\inf_{w\in W}\sup_{v\in V}\dfrac{\mathscr{B}_{\lambda}\left(v,w\right)+\lambda_{0}\left(v,w\right)}{\left\Vert w\right\Vert _{W}\left\Vert v\right\Vert _{V}}\ge \dfrac{\underline{\mu}~\underline{\phi}}{2C_{\phi}}.\label{ineq: norm V}
\end{equation}
From \eqref{ineq: norm W} and \eqref{ineq: norm V}, by the Banach-Necas-Babu\v ska lemma (see \cite{EG2004} or \cite{BF1991}),
for $\lambda_{0}$ large enough,  for any $f\in W'$, there exists a unique solution
$v\in V$ of \eqref{eq: lambda0} and $\left\Vert v\right\Vert _{V}\le C\left\Vert f\right\Vert _{W'}$
for a positive constant $C$.
Hence, our claim is proved.

Now, we fix $\lambda_{0}$ large enough and we define the continuous
linear operator $\overline{R}_{\lambda_{0}}:W'\rightarrow V$ where
$\overline{R}_{\lambda_{0}}\left(f\right)=v$ is the unique solution
of \eqref{eq: lambda0}. Since the injection $\mathcal{I}$ from $V$
to $W'$ is compact, then $\mathcal{I}\circ\overline{R}_{\lambda_{0}}$
is a compact operator from $W'$ into $W'$. By the Fredholm alternative
(see \cite{GT2001}), one of the following assertions holds:
\begin{align}
 & \text{There exists }\overline{v}\in W'\backslash\left\{ 0\right\} \text{ such that }\left(Id-\lambda_{0}\left(\mathcal{I}\circ\overline{R}_{\lambda_{0}}\right)\right)\overline{v}=0.\label{eq: Fredholm 1}\\
 & \text{For any }g\in W',\text{ there exists a unique }\overline{v}\in W'\text{ such that }\left(Id-\lambda_{0}\left(\mathcal{I}\circ\overline{R}_{\lambda_{0}}\right)\right)\overline{v}=g.\label{eq: Fredholm 2}
\end{align}
We claim that \eqref{eq: Fredholm 2} holds. Indeed, assume by contradiction
that \eqref{eq: Fredholm 1} holds. Then there exists $\overline{v}\ne0$
such that $\overline{v}\in V$ and $\mathcal{I}\circ\overline{R}_{\lambda_{0}}\overline{v}=\dfrac{\overline{v}}{\lambda_{0}}$.
Therefore, $ \overline{v}\in V$, and $\mathscr{B}_{\lambda}\left(\dfrac{\overline{v}}{\lambda_{0}},w\right)+\lambda_{0}\left(\dfrac{\overline{v}}{\lambda_{0}},w\right)=\left(\overline{v},w\right)$, for all $w\in W$.
This yields that $\mathscr{B}_{\lambda}\left(\overline{v},w\right)=0$
for all $w\in W$ and by Lemma \ref{lem: linear 1}, we get that $\overline{v}=0$, which leads us to a contradiction. Hence, our claim is proved.

It is then classical to see that  \eqref{eq: Fredholm 2} implies  that  there exists a positive constant $C$ such that 
  for all $f\in W'$,  (\ref{eq: linear 3}) has a unique weak solution $v$  and that
 $\left\Vert v\right\Vert _{V}\le C\left\Vert f\right\Vert _{W'}$, see \cite{khangdao_PhD} for the details.
\end{proof}

\subsection{The Kolmogorov equation}
Consider $b\in PC\left(\Gamma\right)$. This paragraph is devoted to the following boundary value problem including a Kolmogorov equation
\begin{equation}
\begin{cases}
-\mu_{\alpha}\partial^{2}v+b\partial v=0, & \text{in }\Gamma_{\alpha}\backslash\mathcal{V},\alpha\in\mathcal{A},\\
v|_{\Gamma_\alpha} (\nu_i) =v|_{\Gamma_\beta} (\nu_i), \quad &\alpha,\beta \in \cA_i, i\in I
,\\
{\displaystyle \sum_{\alpha\in\mathcal{A}_{i}}\gamma_{i\alpha}\mu_{\alpha}\partial_{\alpha}v\left(\nu_{i}\right)=0,} & i\in I.
\end{cases}\label{eq: linear 2}
\end{equation}
\begin{defn}
\label{def: A star}A weak solution of \eqref{eq: linear 2} is a function
$v\in V$ such that
\[
\mathscr{A}^{\star}\left(v,w\right)=0,\quad\text{for all }w\in W,
\]
where $\mathscr{A}^{\star}:V\times W\rightarrow\mathbb{R}$ is the bilinear
form defined by
\[
\mathscr{A^{\star}}\left(v,w\right):=\sum_{\alpha\in\mathcal{A}}\int_{\Gamma_{\alpha}}\left(\mu_{\alpha}\partial v\partial w+b\partial vw\right)dx.
\]
\end{defn}

As in Remark~\ref{sec:first-class-problems},  if $v$ is a weak solution of $\eqref{eq: linear 2}$, then $v\in C^{2}\left(\Gamma\right)$.

The uniqueness of solutions of (\ref{eq: linear 2}) up to the addition of constants is obtained by using a maximum principle: 
\begin{lem}
\label{lem: linear 2}
For $b\in PC\left(\Gamma\right)$, the solutions of \eqref{eq: linear 2} are the constant functions on $\Gamma$.
\end{lem}

\begin{proof}
[Proof of Lemma \ref{lem: linear 2}]First of all, any constant function on $\Gamma$ is a solution of \eqref{eq: linear 2}.
Now let $v$ be a solution of \eqref{eq: linear 2} then $v\in C^{2}\left(\Gamma\right)$. Assume that the maximum of $v$ over $\Gamma$
is achieved in $\Gamma_{\alpha}$; by the maximum principle,
it is achieved at some endpoint  $\nu_{i}$ of $\Gamma_{\alpha}$.
Without loss of generality, using Remark~\ref{rem:new network}, we can assume that $\pi_{\beta}\left(\nu_{i}\right)=0$
for all $\beta\in\mathcal{A}_{i}$. We have $\partial_{\beta}v\left(\nu_{i}\right)\ge 0$
for all $\beta\in\mathcal{A}_{i}$ because $\nu_{i}$ is the maximum
point of $v$. 
 Since all the coefficients $\gamma_{i\beta},\mu_{\beta}$ are positive,
by the Kirchhoff condition if $\nu_{i}$ is a transition vertex, or by the Neumann boundary condition
 if $\nu_{i}$ is a boundary vertex, we infer that
$\partial_{\beta}v\left(\nu_{i}\right)=0$ for all $\beta\in\mathcal{A}_{i}$.
This implies that $\partial{v_{\beta}}$ is a solution of the first order linear homogeneous differential equation 
$u'+b_{\beta}u=0$, on $\left[0,\ell_{\beta}\right]$, with $u\left(0\right)=0$.
% \[
% \begin{cases}
% u'+b_{\beta}u=0, & \text{on }\left[0,\ell_{\beta}\right],\\
% u\left(0\right)=0.
% \end{cases}
% \]
Therefore, $\partial v_{\beta}\equiv 0$ and $v$ is constant on $\Gamma_{\beta}$ for all $\beta \in\mathcal{A}_{i}$.
We can propagate this argument, starting from the vertices connected to $\nu_i$.
% We have proven the propagation property for the boundary value problem \eqref{eq: linear 2}.
Since the network $\Gamma$ is connected and $v$ is continuous, we obtain that $v$ is constant on $\Gamma$.\end{proof}

\subsection{The dual Fokker-Planck equation}
This paragraph is devoted to the dual boundary value problem of (\ref{eq: linear 2}); it involves a Fokker-Planck equation:
\begin{equation}
\begin{cases}
-\mu_{\alpha}\partial^{2}m-\partial\left(bm\right)=0, & \text{in }\Gamma_{\alpha}\backslash\mathcal{V},\alpha\in\mathcal{A},\\
\dfrac{m|_{\Gamma_{\alpha}}\left(\nu_{i}\right)}{\gamma_{i\alpha}}=\dfrac{m|_{\Gamma_{\beta}}\left(\nu_{i}\right)}{\gamma_{i\beta}}, & \alpha,\beta\in\mathcal{A}_{i},\; i\in I,\\
{\displaystyle \sum_{\alpha\in\mathcal{A}_{i}}\left[   n_{i\alpha} b|_{\Gamma _\alpha}\left(\nu_{i}\right)m|_{\Gamma_{\alpha}}\left(\nu_{i}\right)+\mu_{\alpha}\partial_{\alpha}m\left(\nu_{i}\right)\right]=0}, & 
i\in I,
\end{cases}\label{eq: Fokker-Planck}
\end{equation}
where $b\in PC\left(\Gamma\right)$, 
with
\begin{equation}
m\ge0,\quad\int_{\Gamma}mdx=1.\label{eq: normalization and positive}
\end{equation}
First of all, let $\lambda_0$ be a nonnegative constant; for all $h\in V'$, we introduce the modified boundary value problem
\begin{equation}
\begin{cases}
\lambda_0 m -\mu_{\alpha}\partial^{2}m-\partial\left(bm\right)=h, & \text{in }\Gamma_{\alpha}\backslash\mathcal{V},\alpha\in\mathcal{A},\\
\dfrac{m|_{\Gamma_{\alpha}}\left(\nu_{i}\right)}{\gamma_{i\alpha}}=\dfrac{m|_{\Gamma_{\beta}}\left(\nu_{i}\right)}{\gamma_{i\beta}}, & \alpha,\beta\in\mathcal{A}_{i},\;i\in I,\\
{\displaystyle \sum_{\alpha\in\mathcal{A}_{i}}\left[ n_{i\alpha} b\left(\nu_{i}\right)m|_{\Gamma_{\alpha}}\left(\nu_{i}\right)+\mu_{\alpha}\partial_{\alpha}m\left(\nu_{i}\right)\right]=0}, & i\in I.
\end{cases}\label{eq: Fokker-Planck with lambda_0}
\end{equation}

\begin{defn}
\label{def: bilinear form for FK}
For $\lambda\in \R$, consider the bilinear form $\mathscr{A}_{\lambda}:W\times V\rightarrow\mathbb{R}$ defined by
\[
\mathscr{A}_{\lambda}\left(m,v\right)=\sum_{\alpha\in\mathcal{A}}\int_{\Gamma_{\alpha}}\left[\lambda mv+\left(\mu_{\alpha}\partial m+bm\right)\partial v\right]dx.
\]
A weak solution of \eqref{eq: Fokker-Planck with lambda_0} is a function $m\in W$ such that
\[
\mathscr{A}_{\lambda_{0}}(m,v)=\langle h,v\rangle _{V',V},\quad \text{for all }v\in V.
\]
A weak solution of \eqref{eq: Fokker-Planck} is a function $m\in W$ such
that
\begin{equation}
\mathscr{A}_0(m,v):=\sum_{\alpha\in\mathcal{A}}\int_{\Gamma_{\alpha}}\left(\mu_{\alpha}\partial m+bm\right)\partial vdx=0,\quad\text{for all }v\in V.\label{eq: Fokker-Planck weak}
\end{equation}
\end{defn}

\begin{rem}\label{sec:dual-fokker-planck}
Formally, to get \eqref{eq: Fokker-Planck weak}, we multiply the first line of \eqref{eq: Fokker-Planck}
by $v\in V$, integrate by part, sum over $\alpha\in\mathcal{A}$ 
and use the third line of \eqref{eq: Fokker-Planck} to see that there is no contribution from the vertices.
\end{rem}

\begin{thm}
\label{thm: existence and uniqueness FK}For any $b\in PC\left(\Gamma\right)$,
\begin{itemize}
\item  (Existence) There exists a solution $\widehat{m} \in W$ of \eqref{eq: Fokker-Planck}-\eqref{eq: normalization and positive}
satisfying
\begin{equation}
\left\Vert \widehat{m}\right\Vert _{W}\le C,\quad0\le\widehat{m}\le C,\label{eq: estimate FK}
\end{equation}
where the constant $C$ depends only on $\left\Vert b\right\Vert _{\infty}$
and $\left\{ \mu_{\alpha}\right\} _{\alpha\in A}$. Moreover, $\widehat{m}_{\alpha}\in C^{1}\left(0,\ell_{\alpha}\right)$
for all $\alpha\in\mathcal{A}$. Hence, $\widehat{m}\in\mathcal{W}$.
\item  (Uniqueness) $\widehat{m}$ is the unique solution of \eqref{eq: Fokker-Planck}-\eqref{eq: normalization and positive}.
\item (Strictly positive solution) $\widehat{m}$ is strictly positive.
\end{itemize}
\end{thm}

\begin{proof}
[Proof of existence in Theorem \ref{thm: existence and uniqueness FK}]We
divide the proof of existence into three steps:

\emph{Step 1:} Let $\lambda_{0}$ be a large positive constant that
will be chosen later.
We claim that for $\overline{m}\in  L^2(\Gamma)$ and $h:=\lambda_0 \overline{m} \in L^2(\Gamma)\subset V'$, \eqref{eq: Fokker-Planck with lambda_0} has a unique solution $m\in W$.
This allows us to define a linear operator as follows: 
\begin{displaymath}
  T:L^{2}\left(\Gamma\right)  \longrightarrow W,\quad\quad   T(\overline m)=m,
\end{displaymath}
% \begin{align*}
% T:L^{2}\left(\Gamma\right) & \longrightarrow W,\quad\quad   T(\overline m)=m,
% %\overline{m} & \longmapsto m,
% \end{align*}
where $m$ is the solution of \eqref{eq: Fokker-Planck with lambda_0} with $h=\lambda_0 \overline{m}$. We are
going to prove that $T$ is well-defined and continuous, i.e, for
all $\overline{m}\in L^{2}\left(\Gamma\right)$, \eqref{eq: Fokker-Planck with lambda_0} has a unique
solution that depends continuously on $\overline{m}$. For $w\in W$,
set $\widehat{v}:=w\phi\in V$ where $\phi$ is given by Definition \ref{def: test function}.
We have
\begin{align*}
\mathscr{A}_{\lambda_{0}}\left(w,\widehat{v}\right) & =\sum_{\alpha\in\mathcal{A}}\int_{\Gamma_{\alpha}}\left[\lambda_{0}\phi w^{2}+\left(\mu_{\alpha}\partial w+bw\right)\partial\left(w\phi\right)\right]dx\\
 & =\sum_{\alpha\in\mathcal{A}}\int_{\Gamma_{\alpha}}\left[\left(\lambda_{0}\phi+b\partial\phi\right)w^{2}+\left(\mu_{\alpha}\partial\phi+b\phi\right)w\partial w+\mu_{\alpha}\phi\left(\partial w\right)^{2}\right]dx.
\end{align*}
It follows that when $\lambda_{0}$ is large enough (larger than a constant that only depends on $b,\phi$
and $\mu_{\alpha}$), $\mathscr{A}_{\lambda_{0}}\left(w,\widehat{v}\right)\ge\widehat{C}_{\lambda_{0}}\left\Vert w\right\Vert _{W}^{2}$
for some positive constant $\widehat{C}_{\lambda_{0}}$. Moreover, by
Remark \ref{rem: test function}, there exists a positive constant
$\widehat{C}_{\phi}$ such that for all $w\in W$, we have $\left\Vert w\phi\right\Vert _{V}\le C_{\phi}\left\Vert w\right\Vert _{W}$.
This yields
\[
\inf_{w\in W}\sup_{v\in V}\dfrac{\mathscr{A}_{\lambda_{0}}\left(w,v\right)}{\left\Vert v\right\Vert _{V}\left\Vert w\right\Vert _{W}}\ge\dfrac{\widehat{C}_{\lambda_{0}}}{C_{\phi}}.
\]
Using  similar arguments, for $\lambda_{0}$ large enough, there exist
two positive constants $C_{\lambda_{0}}$ and $C_{\psi}$ such that
\[
\inf_{v\in V}\sup_{w\in W}\dfrac{\mathscr{A}_{\lambda_{0}}\left(w,v\right)}{\left\Vert w\right\Vert _{W}\left\Vert v\right\Vert _{V}}\ge\dfrac{C_{\lambda_{0}}}{C_{\psi}}.
\]
From Banach-Necas-Babu\v ska lemma (see \cite{EG2004} or \cite{BF1991}),
there exists a constant $\overline{C}$ such that for all $\overline{m}\in L^{2}\left(\Gamma\right)$,
there exists a unique solution $m$ of \eqref{eq: Fokker-Planck with lambda_0} with $h=\lambda_0 \overline{m}$
and $\left\Vert m\right\Vert _{W}\le\overline{C}\left\Vert \overline{m}\right\Vert _{L^{2}\left(\Gamma\right)}$.
Hence, the map $T$ is well-defined and continuous from $L^{2}\left(\Gamma\right)$
to $W$.

\emph{Step 2:} 
 Let $K$ be the set defined by
\[
K:=\left\{ m\in L^{2}\left(\Gamma\right):m\ge0\text{ and }\int_{\Gamma}mdx=1\right\} .
\]
We claim that $T\left(K\right)\subset K$ which means
$\int_{\Gamma}m=1$ and $m\ge0$. Indeed, using $v=1$ as a test function
in \eqref{eq: Fokker-Planck with lambda_0}, we have $\int_{\Gamma}mdx=\int_{\Gamma}\overline{m}dx=1$.
Next, consider the negative part $m^-$ of $m$ defined by $m^-(x)=  -\mathds{1}_{\{m(x)<0\}}  m(x)$.
% \[
% m^{-}\left(x\right)=\begin{cases}
% 0 & \text{if }m\left(x\right)\ge0,\\
% -m\left(x\right) & \text{if }m\left(x\right)<0.
% \end{cases}
% \]
Notice that $m^{-}\in W$ and $m^{-}\phi\in V$, where $\phi$
is given by Definition \ref{def: test function}. Using $m^{-}\phi$
as a test function in \eqref{eq: Fokker-Planck with lambda_0} yields
\[
\sum_{\alpha\in\mathcal{A}}-\int_{\Gamma_{\alpha}}\left[\left(\lambda_{0}\phi +b\partial\phi\right)(m^{-})^{2}+\mu_{\alpha}(\partial m^{-})^{2}\phi+\left(\mu_{\alpha}\partial\phi+b\phi\right)m^{-}\partial m^{-}\right]dx=\int_{\Gamma}\lambda_{0}\overline{m}m^{-}\phi dx.
\]
We can see that the right hand side is non-negative. Moreover, for
$\lambda_{0}$ large enough (larger than the same constant as above, which only depends on $b,\phi$ and $\mu_{\alpha}$),
the left hand side is non-positive. This implies that $m^{-}=0$,
and hence $m\ge 0$. Therefore, the claim is proved.

\emph{Step 3:} We claim that $T$ has a fixed point.
Let us now focus on the case when $\overline m \in K$.
 Using
$m\phi$ as a test function in \eqref{eq: Fokker-Planck with lambda_0} yields
\begin{equation}
\sum_{\alpha\in\mathcal{A}}\int_{\Gamma_{\alpha}}\left[\left(\lambda_{0} \phi+b\partial\phi\right)m^{2}+\mu_{\alpha}\left(\partial m\right)^{2}\phi
+\left(\mu_{\alpha}\partial\phi+b\phi\right)m\left(\partial m\right)\right]dx=\int_{\Gamma}\lambda_{0}\overline{m}m\phi dx.\label{eq: estimate fixed point}
\end{equation}
Since $H^{1}\left(0,\ell_{\alpha}\right)$ is continuously embedded
in $L^{\infty}\left(0,\ell_{\alpha}\right)$, there exists a positive
constant $C$ (independent of $\overline m \in K$) such that
\[
\int_{\Gamma}\overline{m}m\phi dx\le\int_{\Gamma}\overline{m}dx\left\Vert m\right\Vert _{L^{\infty}\left(\Gamma\right)}\overline{\phi} = 
\Vert m \Vert _{L^{\infty}\left(\Gamma\right)} \overline{\phi} \le C\left\Vert m\right\Vert _{W}.
\]
Hence, from \eqref{eq: estimate fixed point}, for $\lambda_{0}$
large enough, there exists a positive constant $C_{1}$
such that $C_{1}\left\Vert m\right\Vert _{W}^{2}\le\lambda_{0}C\left\Vert m\right\Vert _{W}$.
Thus
\begin{equation}
\left\Vert m\right\Vert _{W}\le\dfrac{\lambda_{0}C}{C_{1}}.
\end{equation}
Therefore, $T\left(K\right)$ is bounded in $W$. Since the bounded subsets of $W$ are relatively compact
 in $L^{2}\left(\Gamma\right)$, $\overline{T\left(K\right)}$
is compact in $L^{2}\left(\Gamma\right)$. Moreover, we can see that $K$ is closed and convex
in $L^{2}\left(\Gamma\right)$. By Schauder fixed point theorem, see
\cite[Corollary 11.2]{GT2001}, $T$ has a fixed point $\widehat{m}\in K$
which is also a solution of \eqref{eq: Fokker-Planck}
and $\left\Vert \widehat{m}\right\Vert _{W}\le \lambda_{0}C/C_{1}$.

Finally, from the differential equation in \eqref{eq: Fokker-Planck}, for all $\alpha\in\mathcal{A}$, $\left(\widehat{m}_{\alpha}'+b_{\alpha}\widehat{m}_{\alpha}\right)'=0$ on $(0,\ell_{\alpha})$.
Hence, there exists a constant  $C_{\alpha}$ such that 
\begin{equation}
 \label{eq:18}
 \widehat{m}_{\alpha}'+b_{\alpha}\widehat{m}_{\alpha}=C_{\alpha},\quad  \hbox{for all }x\in (0,\ell_{\alpha}).
\end{equation}
It follows that
 $\widehat{m}'_{\alpha}\in C( [0,\ell_{\alpha}])$, for all $\alpha\in\mathcal{A}$.
Hence $\widehat{m}_{\alpha}\in C^{1}([0,\ell_{\alpha}])$
for all $\alpha\in\mathcal{A}$. Thus, $\widehat{m}\in \mathcal{W}$.
\end{proof}
\begin{rem}
\label{rem: smooth solution FK}Let $m\in W$ be a solution of \eqref{eq: Fokker-Planck}.
If $b,\partial b\in PC\left(\Gamma\right)$,  standard arguments yield 
 that $m_{\alpha}\in C^{2} ([0,\ell_{\alpha}])$ for all $\alpha\in\mathcal{A}$. Moreover, by Theorem \ref{thm: existence and uniqueness FK},
there exists a constant $C$ which depends only on $\left\Vert b\right\Vert _{\infty},\left\{ \left\Vert \partial b_{\alpha}\right\Vert _{\infty}\right\} _{\alpha\in\mathcal{A}}$
and $ \mu_{\alpha}$ such that $\left\Vert m_{\alpha}\right\Vert _{C^{2}(0,\ell_{\alpha})}\le C$
for all $\alpha\in\mathcal{A}$. 
\end{rem}
\begin{proof}
[Proof of the positivity in Theorem \ref{thm: existence and uniqueness FK}]
 From
\eqref{eq: normalization and positive}, $\widehat{m}$ is
non-negative on $\Gamma$. Assume by contradiction that there exists
$x_{0}\in\Gamma_{\alpha}$ for some $\alpha\in\mathcal{A}$ such that
$\widehat{m}|_{\Gamma_{\alpha}}\left(x_{0}\right)=0$.
 Therefore, the minimum of $\widehat{m}$ over $\Gamma$ is achieved at 
$x_{0} \in \Gamma_{\alpha}$. 
 If $x_{0}\in \Gamma_{\alpha}\backslash \mathcal{V}$, then $\partial \widehat{m}(x_{0})=0$.
 In (\ref{eq:18}), we thus have $C_{\alpha}=0$, and hence $\widehat{m}_{\alpha}$ satisfies
\[
\widehat{m}_{\alpha}'+b_{\alpha}\widehat{m}_{\alpha}=0,\quad \text{on }\left[0,\ell_{\alpha}\right],\\
\]
with $\widehat{m}_{\alpha}\left(\pi_{\alpha} ^{-1}(x_{0})\right)=0$. It follows that $\widehat m_{\alpha}\equiv 0$
 and $\widehat{m}|_{{\Gamma_{\alpha}}}(\nu_{i})=\widehat{m}|_{{\Gamma_{\alpha}}}(\nu_{j})=0$
 if $\Gamma_\alpha=[\nu_{i},\nu_{j}]$. 

Therefore, it is enough to consider $x_{0} \in \mathcal{V}$.\\
Now, from Remark~\ref{rem:new network},  we may assume without loss of generality that $x_{0}=\nu_{i}$ and $\pi_{\beta}(\nu_{i})=0$ for all $\beta \in \mathcal{A}_i$.  We have the following two cases.

\emph{Case 1:} if $x_{0}=\nu_{i}$  is a transition vertex, then,
since $\widehat{m}$ belongs to $W$, we get 
\begin{equation}
\widehat{m}|_{\Gamma_{\beta}}\left(\nu_{i}\right)=\dfrac{\gamma_{i\beta}}{\gamma_{i\alpha}}\widehat{m}|_{\Gamma_{\alpha}}\left(\nu_{i}\right)=0,\quad\text{for all }\beta\in\mathcal{A}_{i}.\label{eq: jump condition}
\end{equation}
This yields that $\nu_{i}$ is also a minimum point of $\widehat{m}|_{\Gamma_{\beta}}$
for all $\beta\in\mathcal{A}_{i}$. Thus $\partial_{\beta}\widehat{m}\left(\nu_{i}\right)\le0$
for all $\beta\in\mathcal{A}_{i}$. From the transmission condition in \eqref{eq: Fokker-Planck}
which has a classical meaning from the regularity of $\widehat{m}$,
$\partial_{\beta}\widehat{m}\left(\nu_{i}\right)=0$, since all the coefficients
$\mu_{\beta}$ are positive. From (\ref{eq:18}), for all
$\beta\in\mathcal{A}_{i}$, we have
\[
C_{\beta}=\widehat{m}'_{\beta}(0)+b_{\beta}(0)\widehat{m}_{\beta}(0)=0.
\]
Therefore, $\widehat{m}'_{\beta}(y)+b_{\beta}(y)\widehat{m}_{\beta}(y)=0$,
for all $y\in [0,\ell_{\beta}]$ with $\widehat{m}_\beta(0)=0$.
 This implies that $\widehat{m}_{\beta}\equiv 0$ for all $\beta \in \mathcal{A}_i$.
We can propagate the arguments from the vertices connected to $\nu_i$. Since $\Gamma$ is connected,
 we obtain that $\widehat{m}\equiv 0$ on $\Gamma$.

\emph{Case 2:} if $x_{0}=\nu_{i}$  is a boundary vertex, then the 
Robin condition in \eqref{eq: Fokker-Planck} implies that $\partial_{\alpha}\widehat{m}\left(\nu_{i}\right)=0$
since $\mu_{\alpha}$ is positive. From (\ref{eq:18}),
we have $C_{\alpha}=0$.
Therefore, $\widehat{m}'_{\alpha}(y)+b_{\alpha}(y)\widehat{m}_{\alpha}(y)=0$,
for all $y\in [0,\ell_{\alpha}]$ with $\widehat{m}_\alpha(0)=0$. This implies that $\widehat{m}\left(\nu_{j}\right)=0$
where $\nu_{j}$ is the other endpoint of $\Gamma_{\alpha}$.  We are back to Case 1, so $\widehat{m}\equiv 0$ on $\Gamma$.

Finally, we have found that $\widehat{m}\equiv 0$ on $\Gamma$, in contradiction with 
 $\int_{\Gamma}\widehat{m}dx=1$.
\end{proof}
Now we prove uniqueness for \eqref{eq: Fokker-Planck}-\eqref{eq: normalization and positive}.
\begin{proof}
[Proof of uniqueness in Theorem \ref{thm: existence and uniqueness FK}]The
proof of uniqueness is similar to the argument in \cite[Proposition 13]{CM2016}.
As in the proof of Lemma \ref{lem: linear 3}, we can prove that for $\lambda_{0}$ large enough, there exists a constant $C$ such
that for any $f\in V'$, there exists a unique $w\in W$ which satisfies
\begin{equation}
\mathscr{A}_{\lambda_{0}}\left(w,v\right)=\left\langle f,v\right\rangle _{V', V}\text{ for all }v\in V.\label{eq: Fokker-Planck with lambda}
\end{equation}
and $\left\Vert w\right\Vert _{W}\le C\left\Vert f\right\Vert _{V'}$.
This allows us to define the continuous linear operator
\begin{align*}
S_{\lambda_{0}}:L^{2}\left(\Gamma\right) & \longrightarrow W,\\
f & \longmapsto w,
\end{align*}
where $w$ is a solution of \eqref{eq: Fokker-Planck with lambda}.
Then we define $R_{\lambda_{0}}=\mathcal{J}\circ S_{\lambda_{0}}$
where $\mathcal{J}$ is the injection from $W$ in $L^{2}\left(\Gamma\right)$,
which is compact. Obviously, $R_{\lambda_{0}}$ is a compact operator
from $L^{2}\left(\Gamma\right)$ into $L^{2}\left(\Gamma\right)$. Moreover,
$m\in W$ is a solution of \eqref{eq: Fokker-Planck} if and only if $m\in\ker\left(Id-\lambda_{0}R_{\lambda_{0}}\right)$.
By Fredholm alternative, see \cite{GT2001}, $\dim\ker\left(Id-\lambda_{0}R_{\lambda_{0}}\right)=\dim\ker\left(Id-\lambda_{0}R_{\lambda_{0}}^{\star}\right)$.

In order to characterize $R_{\lambda_{0}}^{\star}$, we now consider
the following boundary value problem for $g\in L^2(\Gamma)\subset W'$:
\begin{equation}
\begin{cases}
\lambda_{0}v-\mu_{\alpha}\partial^{2}v+b\partial v=g, & \text{in }\Gamma_{\alpha}\backslash\mathcal{V},\alpha\in\mathcal{A},\\
v|_{\Gamma_\alpha} (\nu_i) =v|_{\Gamma_\beta} (\nu_i) \quad & \alpha,\beta \in \cA_i,\; i\in I, \\
{\displaystyle \sum_{\alpha\in\mathcal{A}_{i}}\gamma_{i\alpha}\mu_{\alpha}\partial_{\alpha}v\left(\nu_{i}\right)=0,} & i\in I.
\end{cases}\label{eq: adjoint problem}
\end{equation}
A weak solution of \eqref{eq: adjoint problem} is a function $v\in V$
such that
\[
\mathscr{T}_{\lambda_{0}}\left(v,w\right):=\sum_{\alpha\in\mathcal{A}}\int_{\Gamma_{\alpha}}(\lambda_{0}vw+\mu_{\alpha}\partial v\partial w+bw\partial v)dx=\int_{\Gamma}gwdx,\quad\text{for all }w\in W.
\]
Using similar arguments as in the proof of existence in
Theorem \ref{thm: existence and uniqueness FK}, we see that for $\lambda_0$ large enough and 
all $g\in L^{2}\left(\Gamma\right)$, there exists a unique solution
$v\in V$ of \eqref{eq: adjoint problem}. Moreover, there exists
a constant $C$ such that $\left\Vert v\right\Vert _{V}\le C\left\Vert g\right\Vert _{L^{2}\left(\Gamma\right)}$
for all $g\in L^{2}\left(\Gamma\right)$. This allows us to define
a continuous operator 
\begin{align*}
T_{\lambda_{0}}:L^{2}\left(\Gamma\right) & \longrightarrow V,\\
g & \longmapsto v.
\end{align*}
Then we define $\tilde{R}_{\lambda_{0}}=\mathcal{I}\circ T_{\lambda_{0}}$
where $\mathcal{I}$ is the injection from $V$ in $L^{2}\left(\Gamma\right)$.
Since $\mathcal{I}$  compact, $\tilde{R}_{\lambda_{0}}$ is a compact
operator from $L^{2}\left(\Gamma\right)$ into $L^{2}\left(\Gamma\right)$.
For any $g\in L^2(\Gamma)$, set $v=T_{\lambda_{0}}g$.
 Noticing that $\mathscr{T}_{\lambda_{0}}(v,w)=\mathcal{A}_{\lambda_{0}}(w,v)$ for all $v\in V,w\in W$, we obtain            that
\[
\left(g,R_{\lambda_{0}}f\right)_{L^{2}\left(\Gamma\right)}%=\left(g,\mathcal{J}\circ S_{\lambda_{0}}f\right)_{L^{2}\left(\Gamma\right)}
=\mathscr{T}_{\lambda_{0}}\left(v,S_{\lambda_{0}}f\right)=\mathscr{A}_{\lambda_{0}}\left(S_{\lambda_{0}}f,v\right)=\left(f,v\right)_{L^{2}\left(\Gamma\right)}=(f,\tilde{R}_{\lambda_{0}}g)_{L^{2}\left(\Gamma\right)}.
\]
Thus $R_{\lambda_{0}}^{\star}=\tilde{R}_{\lambda_{0}}$. 
But $\ker\left(Id-\lambda_{0}\tilde R_{\lambda_{0}}\right)$ is the set of solutions of (\ref{eq: linear 2}), which, from Lemma~\ref{lem: linear 2}, consists of constant functions on $\Gamma$. 
This implies that $\dim \ker\left(Id-\lambda_{0}R_{\lambda_{0}}^{\star}\right) =1$ and then that
$\dim\ker\left(Id-\lambda_{0}R_{\lambda_{0}}\right)=\dim\ker\left(Id-\lambda_{0}R_{\lambda_{0}}^{\star}\right)=1$.
Finally, since the solutions $m$ of \eqref{eq: Fokker-Planck} are in $\ker\left(Id-\lambda_{0}R_{\lambda_{0}}\right)$ and satisfy 
 the normalization condition $\int_{\Gamma}mdx=1$, we obtain the desired uniqueness property in Theorem~\ref{thm: existence and uniqueness FK}.
\end{proof}

\section{\label{sec: H-J equation and ergodic problem}Hamilton-Jacobi equation
and the ergodic problem}

\subsection{\label{subsec: H-J equation}The Hamilton-Jacobi equation}

This section is devoted to the following boundary value problem
including a Hamilton-Jacobi equation:
\begin{equation}
\begin{cases}
-\mu_{\alpha}\partial^{2}v+H\left(x,\partial v\right)+\lambda v=0, & \text{in }\Gamma_{\alpha}\backslash\mathcal{V},\alpha\in A,\\
v|_{\Gamma_\alpha} (\nu_i) =v|_{\Gamma_\beta} (\nu_i), \quad & \alpha,\beta \in \cA_i, \; i\in I,
\\
{\displaystyle \sum_{\alpha\in\mathcal{A}_{i}}}\gamma_{i\alpha}\mu_{\alpha}\partial_{\alpha}v\left(\nu_{i}\right)=0, & i\in I,
\end{cases}\label{eq:HJ}
\end{equation}
where $\lambda$ is a positive constant and the Hamiltonian $H:\Gamma\times\mathbb{R}\rightarrow\mathbb{R}$
is defined in Section \ref{sec: Mean-field-games}, except that, in \eqref{eq:HJ} and the whole Section \ref{subsec: H-J equation}
below, the Hamiltonian contains the coupling term, i.e, $H\left(x,\partial v\right)$
in \eqref{eq:HJ} plays the role of $H\left(x,\partial v\right)-F\left(m\left(x\right)\right)$
in \eqref{eq: MFG system}.

\begin{defn}
  \begin{itemize}

\item  A classical solution of \eqref{eq:HJ} is a function $v\in C^{2}\left(\Gamma\right)$
which satisfies \eqref{eq:HJ} pointwise.

\item  A weak solution of \eqref{eq:HJ} is a function $v\in V$ such
that
\[
\sum_{\alpha\in\mathcal{A}}\int_{\Gamma_{\alpha}}\left(\mu_{\alpha}\partial v\partial w+H\left(x,\partial v\right)w+\lambda vw\right)dx=0\quad\text{for all }w\in W.
\]
\end{itemize}
\end{defn}

\begin{prop}
\label{prop: main}Assume that
\begin{align}
 & H_{\alpha}\in C\left(\left[0,\ell_{\alpha}\right]\times \mathbb{R}\right),\label{eq: continuity H}\\
 & \left|H\left(x,p\right)\right|\le C_{2}\left(1+\left|p\right|^{2}\right)\text{ for all }x\in\Gamma,p\in\mathbb{R},\label{eq: bounded H}
\end{align}
where $C_{2}$ is a positive constant. There exists a classical solution
$v$ of \eqref{eq:HJ}. Moreover, if $H_{\alpha}$ is locally
Lipschitz with respect to both variables for all $\alpha\in\mathcal{A}$, 
then the solution $v$ belongs to $C^{2,1}\left(\Gamma\right)$.
\end{prop}
% Let us introduce additional Sobolev spaces on $\Gamma$.
% \begin{defn} \label{def: Sobolev space}
% For any integer $p\ge1$, the Sobolev space $H^p(\Gamma)$ is defined as follows
% \[
% H^p(\Gamma):=\left\{ v\in C\left(\Gamma\right):v_{\alpha}\in H^{p}\left(0,\ell_{\alpha}\right)\text{ for all }\alpha\in\mathcal{A}\right\},
% \]
% and endowed with the norm
% \[
% \left\Vert v\right\Vert _{H^{p}\left(\Gamma\right)}=\sqrt{\sum^{p}_{k=1}\sum_{\alpha\in\mathcal{A}}
% \left\Vert \partial^{k}v_{\alpha}\right\Vert _{L^{2}\left(0,\ell_{\alpha}\right)}^{2}+
% \left\Vert v\right\Vert _{L^2(\Gamma)}^{2}}.
% \]
% For $p=1$, the Sobolev space $H^1(\Gamma)$ coincides with $V$.
% \end{defn}
\begin{rem}\label{rem: regularity for HJ}
Assume \eqref{eq: continuity H} and that $v\in H^2(\Gamma)\subset V$ is a  weak solution of \eqref{eq:HJ}. From the compact embedding of $H^{2}\left(0,\ell_{\alpha}\right)$ into $C^{1,\sigma}([0,\ell_{\alpha}])$ for all 
$\sigma\in (0,1/2)$, we get $v\in C^{1,\sigma}(\Gamma)$. Therefore, from the differential equation in \eqref{eq:HJ}
$
\mu_{\alpha}\partial ^2 v_{\alpha}(\cdot)= H_\alpha(\cdot,\partial v_{\alpha}(\cdot))+\lambda v_{\alpha}(\cdot)\in C([0,\ell_{\alpha}])$.
It follows that $v$ is a classical solution of \eqref{eq:HJ}.
\end{rem}

\begin{rem}\label{sec:hamilt-jacobi-equat}
  Assume now  that $H$ is locally Lipschitz continuous and that $v\in H^2(\Gamma)\subset V$ is a  weak solution of \eqref{eq:HJ}.
 From Remark~\ref{rem: regularity for HJ}, $v \in C^{1,\sigma}(\Gamma)$ for $\sigma\in (0,1/2)$ and the function $ 
-\lambda v_{\alpha}-H_{\alpha}\left(\cdot,\partial v_{\alpha}\right)$ belongs to $C^{0,\sigma}([0,\ell_\alpha])$.
Then, from the first line of  \eqref{eq:HJ}, $v \in C^{2,\sigma}(\Gamma)$. This implies that $\partial v_\alpha\in {\rm{Lip}}[0,\ell_\alpha]$ and using the differential equation again, we see that $v\in C^{2,1}(\Gamma)$. 
\end{rem}

Let us start with the case when $H$ is a bounded Hamiltonian.
\begin{lem}
\label{lem: H is bounded}Assume \eqref{eq: continuity H} and for
some $C_{H}>0$,
\begin{equation}
\left|H\left(x,p\right)\right|\le C_{H},\quad\text{for all }\left(x,p\right)\in\Gamma\times\mathbb{R}.
\end{equation}
There exists a classical solution $v$ of \eqref{eq:HJ}. Moreover, if
$H_{\alpha}$ is locally Lipschitz in $[0,\ell_\alpha]\times \R$ for
all $\alpha\in\mathcal{A}$ then the solution $v$ belongs to $C^{2,1}\left(\Gamma\right)$.
\end{lem}

\begin{proof}
[Proof of Lemma \ref{lem: H is bounded}]For any $u\in V$, from Lemma
\ref{lem: linear 3}, the following boundary value problem:
\begin{equation}
\begin{cases}
-\mu_{\alpha}\partial^{2}v+\lambda v=-H\left(x,\partial u\right), & \text{if }x\in\Gamma_{\alpha}\backslash\mathcal{V},\alpha\in\mathcal{A},\\
v|_{\Gamma_\alpha} (\nu_i) =v|_{\Gamma_\beta} (\nu_i),\quad & \alpha,\beta \in \cA_i,\; i\in I,
\\
{\displaystyle \sum_{\alpha\in\mathcal{A}_{i}}}\gamma_{i\alpha}\mu_{\alpha}\partial_{\alpha}v\left(\nu_{i}\right)=0, &  i\in I,
\end{cases}\label{eq: fixed point}
\end{equation}
has a unique weak solution $v\in V$. This allows us to define the
map $T: V  \longrightarrow V$ by $T(u):=v$.
% \begin{align*}
% T:V & \longrightarrow V,\\
% u & \longmapsto v.
% \end{align*}
Moreover, from Lemma \ref{lem: linear 3}, there exists a constant $C$ such that
\begin{equation}
\left\Vert v\right\Vert _{V}\le C\left\Vert H\left(x,\partial u\right)\right\Vert _{L^{2}\left(\Gamma\right)}\le CC_{H}\left|\Gamma\right|^{1/2},\label{eq: boundedness v in V}
\end{equation}
where
$|\Gamma|=\Sigma_{\alpha\in\mathcal{A}}\ell_{\alpha}$. Therefore, from the differential equation in \eqref{eq: fixed point},
\begin{equation}
\underline{\mu}\left\Vert \partial^{2}v\right\Vert _{L^{2}\left(\Gamma\right)}\le\lambda\left\Vert v\right\Vert _{L^{2}\left(\Gamma\right)}+\left\Vert H\left(x,\partial u\right)\right\Vert _{L^{2}\left(\Gamma\right)}\le \lambda\left\Vert v\right\Vert _{V}+C_{H}\left|\Gamma\right|^{1/2} \le \left(\lambda C+1\right)C_{H}\left|\Gamma\right|^{1/2},\label{eq: boundedness v" in L2}
\end{equation}
where $\underline{\mu}:=\min_{\alpha\in\mathcal{A}}\mu_{\alpha}$.
From \eqref{eq: boundedness v in V} and \eqref{eq: boundedness v" in L2}, $T\left(V\right)$ is a bounded subset of $H^{2}\left(\Gamma\right)$, see Definition \ref{def: Sobolev space}.
From the compact embedding of $H^{2}\left(\Gamma\right)$ into $V$,
we deduce that $\overline{T\left(V\right)}$ is a compact subset of
$V$. 

Next, we claim that $T$ is continuous from $V$ to $V$. Assuming that
\begin{equation}
\begin{cases}
u_{n}\rightarrow u, & \text{in \ensuremath{V},}\\
v_{n}=T\left(u_{n}\right), & \text{for all \ensuremath{n},}\\
v=T\left(u\right),
\end{cases}\label{eq: converge in V}
\end{equation}
we need to prove that $v_{n}\rightarrow v$ in $V$. Since $\left\{ v_{n}\right\} $
is uniformly bounded in $H^{2}\left(\Gamma\right)$, then, up to the
extraction of a subsequence, $v_{n}\rightarrow\widehat{v}$ in $C^{1,\sigma}\left(\Gamma\right)$ for some $\sigma\in (0,1/2)$. From \eqref{eq: converge in V}, we have that $\partial u_{n}\rightarrow\partial u$
in $L^{2}\left(\Gamma_{\alpha}\right)$ for all $\alpha\in\mathcal{A}$.
This yields that, up to another extraction of a subsequence, $\partial u_{n}\rightarrow\partial u$
almost everywhere in $\Gamma_{\alpha}$. Thus $H\left(x,\partial u_{n}\right)\rightarrow H\left(x,\partial u\right)$
in $L^{2}\left(\Gamma_{\alpha}\right)$ by Lebesgue dominated convergence
theorem. Hence,  $\widehat{v}$ is a weak solution of \eqref{eq: fixed point}.
Since the latter is  unique,  $\widehat{v}=v$ and we can conclude that the whole sequence $ v_{n} $
converges to $v$. The claim is proved.

From  Schauder fixed point theorem,
see \cite[Corollary 11.2]{GT2001}, $T$ admits a fixed point which
is a weak solution of \eqref{eq:HJ}. Moreover, recalling that $v\in H^2(\Gamma)$, 
 we obtain that $v$ is a classical solution of \eqref{eq:HJ} from Remark \ref{rem: regularity for HJ}.

Assume now that $H$ is locally Lipschitz. 
 Since  $v_{\alpha}\in H^{2}\left(0,\ell_{\alpha}\right)$
 for all $\alpha\in\mathcal{A}$, we may use Remark \ref{sec:hamilt-jacobi-equat}
and obtain that $v\in C^{2,1}\left(\Gamma\right)$.
%  we obtain that 
%  $v\in \hbox{Lip}(\Gamma)$ and $\partial v_{\alpha}\in
%   C^{0,\sigma}\left(\left[0,\ell_{\alpha}\right]\right)$  for all $\sigma\in (0,1/2) $.
% From \eqref{eq:HJ}, we get $\partial^{2}v_{\alpha}\in  C^{0,\sigma}\left(\left[0,\ell_{\alpha}\right]\right)$
% for all $\alpha\in\mathcal{A}$, and hence $v\in C^{2,\sigma}\left(\Gamma\right)$. 
%  In turn, this implies that $\partial v_{\alpha}\in \hbox{Lip}\left(\left[0,\ell_{\alpha}\right]\right)$,  and 
%  from \eqref{eq:HJ} that $v\in C^{2,1}\left(\Gamma\right)$.
\end{proof}
\begin{lem}
\label{lem: maximum principle} If $v,u\in C^{2}\left(\Gamma\right)$
satisfy
\begin{equation}
\begin{cases}
-\mu_{\alpha}\partial^{2}v+H\left(x,\partial v\right)+\lambda v\ge-\mu_{\alpha}\partial^{2}u+H\left(x,\partial u\right)+\lambda u, & \text{if }x\in\Gamma_{\alpha}\backslash\mathcal{V},\alpha\in A,\\
{\displaystyle \sum_{\alpha\in\mathcal{A}_{i}}\gamma_{i\alpha}\mu_{\alpha}\partial_{\alpha}v\left(\nu_{i}\right)\ge\sum_{\alpha\in\mathcal{A}_{i}}\gamma_{i\alpha}\mu_{\alpha}\partial_{\alpha}u\left(\nu_{i}\right)}, & \text{if }\nu_{i}\in\mathcal{V},
\end{cases}\label{eq: maximum principle}
\end{equation}
then $v\ge u$.
\end{lem}

\begin{proof}
[Proof of Lemma \ref{lem: maximum principle}] The proof is reminiscent of an argument in  \cite{CMS2013}. Suppose by contradiction
that $\delta:=\max_\Gamma \left\{ u-v\right\} >0$. Let $x_{0}\in\Gamma_{\alpha}$
be a maximum point of $u-v$. It suffices to consider the case when $x_0 \in \mathcal{V}$, since if $x_0 \in \Gamma \backslash \mathcal{V}$, then 
$
u\left(x_{0}\right)>v\left(x_{0}\right)$, $\partial u\left(x_{0}\right)=\partial v\left(x_{0}\right)$,
$\partial^{2}u\left(x_{0}\right)\le\partial^{2}v\left(x_{0}\right)$,
and we obtain a contradiction with the first line of \eqref{eq: maximum principle}.
\\
Now consider the case when $x_{0}=\nu_{i}\in \mathcal{V}$; from Remark~\ref{rem:new network}, we
can assume without restriction that $\pi_{\alpha}\left(0\right)=\nu_{i}$.
Since $u-v$ achieves its maximum over $\Gamma$ at $\nu_{i}$, we obtain
that
$
\partial_{\beta}u\left(\nu_{i}\right) \ge\partial_{\beta}v\left(\nu_{i}\right)$, for all $\beta\in\mathcal{A}_{i}
$.
From Kirchhoff conditions in \eqref{eq: maximum principle}, this
implies that
$
\partial_{\beta}u\left(\nu_{i}\right)=\partial_{\beta}v\left(\nu_{i}\right)$, for all $\beta\in\mathcal{A}_{i}$.
It follows that $\partial v_{\alpha}(0)= \partial u_{\alpha}(0)$. Using the first line of \eqref{eq: maximum principle}, we get that
\[
-\mu_{\alpha}\left[\partial^{2}v_{\alpha}(0)-\partial^{2}u_{\alpha}(0)\right]\ge\underset{=0}{\underbrace{H_{\alpha}\left(0,\partial u_{\alpha}(0)\right)-H_\alpha\left(0,\partial v_{\alpha}(0\right))}}+\lambda\left(u_{\alpha}(0)-v_{\alpha}(0)\right)>0.
\]
Therefore, $u_{\alpha}-v_{\alpha}$ is locally strictly convex in $[0,\ell_{\alpha}]$
near $0$ and its first order derivative vanishes at $0$. This contradicts
the fact that $\nu_{i}$ is the maximum point of $u-v$.
\end{proof}
We now turn to Proposition \ref{prop: main}.
\begin{proof}
[Proof of Proposition \ref{prop: main}]We adapt the classical proof of Boccardo,
Murat and Puel in \cite{BMP1983}. First of all, we truncate the Hamiltonian
as follows:
\[
H_{n}\left(x,p\right)=\begin{cases}
H\left(x,p\right), & \text{if }\left|p\right|\le n,\\
H\left(x,\dfrac{p}{\left|p\right|}n\right), & \text{if }\left|p\right|>n.
\end{cases}
\]
By Lemma \ref{lem: H is bounded}, for all $n\in\mathbb{N}$, since
$H_{n}\left(x,p\right)$ is continuous and bounded by $C_{2}\left(1+n^{2}\right)$,
there exists a classical  solution $v_{n}\in C^{2}\left(\Gamma\right)$
for the following boundary value problem 
\begin{equation}
\begin{cases}
-\mu_{\alpha}\partial^{2}v+H_{n}\left(x,\partial v\right)+\lambda v=0, & x\in\Gamma_{\alpha}\backslash\mathcal{V},\alpha\in A,\\
v|_{\Gamma_\alpha} (\nu_i) =v|_{\Gamma_\beta} (\nu_i),\quad &  \hbox{ for all } \alpha,\beta \in \cA_i, \; i\in I,
\\
{\displaystyle \sum_{\alpha\in\mathcal{A}_{i}}}\gamma_{i\alpha}\mu_{\alpha}\partial_{\alpha}v\left(\nu_{i}\right)=0, &  i\in I.
\end{cases}\label{eq: truncated}
\end{equation}
We wish to pass to the limit as $n$ tend to $+\infty$;  we first need to estimate
$v_{n}$  uniformly in $n$, successively in $L^{\infty}\left(\Gamma\right)$, $H^{1}\left(\Gamma\right)$
and $H^{2}\left(\Gamma\right)$.

\emph{Estimate in }\textbf{$L^{\infty}\left(\Gamma\right)$}. Since $\left|H_{n}\left(x,p\right)\right|\le c\left(1+\left|p\right|^{2}\right)$
for all $x,p$, then $\varphi=-c/\lambda$ and $\overline{\varphi}=c/\lambda$
are  respectively a sub- and super-solution of \eqref{eq: truncated}.
Therefore, from Lemma \ref{lem: maximum principle}, we obtain $|\lambda v_n|\le c$.

\emph{Estimate in}\textbf{ $V$}. For a positive constant $K$ to be chosen later, 
we introduce $w_{n}:=e^{Kv_{n}^{2}}v_{n}\psi \in W$,
where $\psi$ is given in  Definition \ref{def: test function}.
 Using $w_{n}$ as a test function in \eqref{eq: truncated} leads to
\[
\sum_{\alpha\in\mathcal{A}}\int_{\Gamma_{\alpha}}\left(\mu_{\alpha}\partial v_{n}\partial w_{n}+\lambda v_{n}w_{n}\right)dx=-\int_{\Gamma}H_{n}\left(x,\partial v_{n}\right)w_{n}dx.
\]
Since $\left|H_{n}\left(x,p\right)\right|\le c\left(1+p^{2}\right)$,
we have
\begin{align*}
 & \sum_{\alpha\in\mathcal{A}}\int_{\Gamma_{\alpha}}e^{Kv_{n}^{2}}\left[\left(\mu_{\alpha}\psi\right)\left(\partial v_{n}\right)^{2}+\left(\mu_{\alpha}2K\psi\right)v_{n}^{2}\left(\partial v_{n}\right)^{2}+\left(\mu_{\alpha}\partial\psi\right)v_{n}\partial v_{n}+\lambda \psi v_{n}^{2}\right]dx \\
\le & \int_{\Gamma}e^{Kv_{n}^{2}}\left|H_{n}\left(x,\partial v_{n}\right)\right|\left|v_{n}\psi\right|dx\\
\le & \int_{\Gamma}c e^{Kv_{n}^{2}}\psi\left|v_{n}\right| dx+\int_{\Gamma}c\psi e^{Kv_{n}^{2}}\left|v_{n}\right|\psi\left(\partial v_{n}\right)^{2}dx\\
\le & \int_{\Gamma}e^{Kv_{n}^{2}}\left( \lambda \psi v_{n}^2 +\psi\dfrac{c^2}{4\lambda}\right) dx+\sum_{\alpha\in\mathcal{A}}\int_{\Gamma_{\alpha}}e^{Kv_{n}^{2}}\left[\dfrac{\mu_{\alpha}}{2}\psi\left(\partial v_{n}\right)^{2}+\dfrac{c^2}{2\mu_{\alpha}}\psi\left(\partial v_{n}\right)^{2}v_{n}^{2}\right]dx,
\end{align*}
where we have used Young inequalities. Since $\lambda>0$ and $\psi>0$, we deduce that
\begin{equation}
\label{eq:19}
\begin{split}
 &\sum_{\alpha\in\mathcal{A}}\int_{\Gamma_{\alpha}}e^{Kv_{n}^{2}}\left[\left(\dfrac{\mu_{\alpha}}{2}\psi\right)
\left(\partial v_{n}\right)^{2}+2\psi\left(\mu_{\alpha}K-\dfrac{c^2}{4\mu_{\alpha}}\right)v_{n}^{2}\left(\partial v_{n}\right)^{2}+
\left(\mu_{\alpha}\partial\psi\right)v_{n}\partial v_{n}\right]dx \\ \le &\dfrac{c^2}{4\lambda} \int_{\Gamma}e^{Kv_{n}^{2}}\psi dx.  
\end{split}
%\label{eq: estimate dv_n in L^2}
\end{equation} 
%  Since $v_{n}$ is bounded by $c/\lambda$,  there exists a constant $C$ independent of $n$ such that  
% \begin{equation*}
%  \sum_{\alpha\in\mathcal{A}}\int_{\Gamma_{\alpha}}e^{Kv_{n}^{2}}\left[2\psi\left(\mu_{\alpha}K-\dfrac{c^2}{\mu_{\alpha}}\right)v_{n}^{2}\left(\partial v_{n}\right)^{2}+\left(\mu_{\alpha}\partial\psi\right)v_{n}\partial v_{n}\right]dx \le C.
% \end{equation*}
Next, choosing $K>     (1+ c^2/ 4 \underline{\mu}) /\underline{\mu} $ yields that 
\begin{equation*}
 \sum_{\alpha\in\mathcal{A}}\int_{\Gamma_{\alpha}}e^{Kv_{n}^{2}}\left[   \dfrac{\mu_{\alpha}}{2}\psi
\left(\partial v_{n}\right)^{2}+      2\psi v_{n}^{2}\left(\partial v_{n}\right)^{2}+\left(\mu_{\alpha}\partial\psi\right)v_{n}\partial v_{n}\right]dx \le C
\end{equation*}
for a positive constant $C$ independent of $n$, because $v_{n}$ is bounded by $c/\lambda$.
Since $\psi $ is bounded from below by a positive number and  $\partial \psi$ is piecewise constant on $\Gamma$,
we  infer that \\
$\sum_{\alpha\in\mathcal{A}}\int_{\Gamma_{\alpha}}e^{Kv_{n}^{2}}v_{n}^{2}\left(\partial v_{n}\right)^{2} \le \widetilde C$,
where $\widetilde C$ is a positive constant independent on $n$.
Using this information and (\ref{eq:19})
 again, we obtain that $\int_{\Gamma}\left(\partial v_{n}\right)^{2}$ is bounded uniformly in $n$. There
exists a constant $\overline{C}$ such that $\left\Vert v_{n}\right\Vert _{V}\le\overline{C}$ for all $n$.

\emph{Estimate in}\textbf{ $H^{2}\left(\Gamma\right)$}. From the differential equation  in \eqref{eq: truncated} and \eqref{eq: bounded H}, we have
\[
\underline{\mu}\left|\partial^{2}v_{n}\right|\le c+c\left|\partial v_{n}\right|^{2}+\lambda\left|v_{n}\right|,\quad\text{for all \ensuremath{\alpha\in\mathcal{A}}}.
\]
Thus $ \partial^{2}v_{n} $ is uniformly bounded in
$L^{1}\left(\Gamma\right)$. This and the previous estimate on $\|\partial v_n\|_{L^2 (\Gamma)}$ yield that 
 $ \partial v_{n} $ is uniformly bounded in $L^{\infty}\left(\Gamma\right)$,
 from the continuous embedding of $W^{1,1}\left(0,\ell_{\alpha}\right)$ into $C\left(\left[0,\ell_{\alpha}\right]\right)$.
Therefore, from \eqref{eq: truncated}, we get that $ \partial^{2}v_{n} $
is uniformly bounded in $L^{\infty}\left(\Gamma\right)$. This implies in particular
that $ v_{n} $ is uniformly bounded in $W^{2,\infty}\left(\Gamma\right)$.

Hence, for any $\sigma \in (0,1)$, up to the extraction of a subsequence, there exists $v\in V$
such that $v_{n}\rightarrow v$ in $C^{1,\sigma}\left(\Gamma\right)$. This
yields that $H_{n}\left(x,\partial v_{n}\right)\rightarrow H\left(x,\partial v\right)$
for all $x\in\Gamma$. By Lebesgue's Dominated Convergence Theorem,
we obtain that $v$ is a weak solution of \eqref{eq:HJ}, and since $v\in C^{1,\sigma}(\Gamma)$, by Remark 
\ref{rem: regularity for HJ},  $v$ is a classical solution of \eqref{eq:HJ}.

Assume now that $H$ is locally Lipschitz.  We may use Remark \ref{sec:hamilt-jacobi-equat}
and obtain that $v\in C^{2,1}\left(\Gamma\right)$.
% Assume now  that $H$ is locally Lipschitz continuous. Since $v \in C^{1,\sigma}(\Gamma)$, the function $ 
% -\lambda v_{\alpha}-H_{\alpha}\left(\cdot,\partial v_{\alpha}\right)$ belongs to $C^{0,\sigma}(\Gamma)$.
% Since $-\mu_{\alpha}\partial^{2}v_{\alpha}=-\lambda v_{\alpha}-H_{\alpha}\left(\cdot,\partial v_{\alpha}\right)$
% in $\left(0,\ell_{\alpha}\right)$, we infer that $v\in C^{2, \sigma}(\Gamma)$. Then $\partial v_\alpha\in {\rm{Lip}}[0,\ell_\alpha]$ and using the PDE again, we see that $v\in C^{2,1}(\Gamma)$. 
 The proof is complete.
%  We already know that 
%  We claim that $v_{\alpha}\in C^{2,\sigma}\left(\left[0,\ell_{\alpha}\right]\right)$
% for any $\alpha\in\mathcal{A}$. Indeed, for fixed $\alpha$, the
% equation for $v$ in the distributional sense is
% \begin{equation}
% -\mu_{\alpha}\partial^{2}v_{\alpha}=-\lambda v_{\alpha}-H_{\alpha}\left(\cdot,\partial v_{\alpha}\right),\quad\text{in }\left(0,\ell_{\alpha}\right).\label{eq: distribution sense}
% \end{equation}
% Since $v_{\alpha}\in C\left(\left[0,\ell_{\alpha}\right]\right)$
% and $H_{\alpha}\left(\cdot,\partial v_{\alpha}\left(\cdot\right)\right)\in L^{1}\left(0,\ell_{\alpha}\right)$,
% we get $v_{\alpha}\in W^{2,1}\left(0,\ell_{\alpha}\right)$ and therefore
% $\partial v_{\alpha}\in L^{p}\left(0,\ell_{\alpha}\right)$ for any
% $p\ge1$. Hence $H_{\alpha}\left(\cdot,\partial v_{\alpha}\left(\cdot\right)\right)\in L^{p}\left(0,\ell_{\alpha}\right)$
% for any $p\ge q$, where $q$ appear in Assumption \eqref{eq: derivative of H is bounded}, and by \eqref{eq: distribution sense}, $v_{\alpha}\in W^{2,p}\left(0,\ell_{\alpha}\right)$.
% Finally by \eqref{eq: distribution sense}, $\partial^{2} v_{\alpha}\in C^{0,\sigma}\left(0,\ell_{\alpha}\right)$, which ends the proof.
\end{proof}

\subsection{The ergodic problem\label{subsec:Ergodic-problem}}

For $f\in PC\left(\Gamma\right)$,
we wish to prove the existence of  $\left(v,\rho\right)\in C^{2}\left(\Gamma\right)\times\mathbb{R}$
such that
\begin{equation}
\begin{cases}
-\mu_{\alpha}\partial^{2}v+H\left(x,\partial v\right)+\rho=f\left(x\right), & \text{in }\Gamma_{\alpha}\backslash\mathcal{V},\alpha\in\mathcal{A},\\
v|_{\Gamma_\alpha} (\nu_i) =v|_{\Gamma_\beta} (\nu_i),\quad &\alpha,\beta \in \cA_i, \;i\in I,\\
{\displaystyle \sum_{\alpha\in\mathcal{A}_{i}}\gamma_{i\alpha}\mu_{\alpha}\partial_{\alpha}v\left(\nu_{i}\right)=0}, & i\in I,
\end{cases}\label{eq: ergodic}
\end{equation}
with the normalization condition
\begin{equation}
\int_{\Gamma}vdx=0.\label{eq: normalization}
\end{equation}

\begin{thm}
\label{thm: existence ergodic}Assume \eqref{eq: Hamiltonian}-\eqref{eq: super quadratic}.
There exists a unique couple $\left(v,\rho\right)\in C^{2}\left(\Gamma\right)\times\mathbb{R}$
satisfying \eqref{eq: ergodic}-\eqref{eq: normalization}, with $\left|\rho\right|\le\max_{x\in\Gamma}\left|H\left(x,0\right)-f\left(x\right)\right|$. There exists a constant $\overline{C}$ which only depends upon $\left\Vert f\right\Vert _{L^{\infty}\left(\Gamma\right)},\mu_{\alpha}$
 and the constants in \eqref{eq: super quadratic}
such that
\begin{equation}
\left\Vert v\right\Vert _{C^{2}\left(\Gamma\right)}\le\overline{C}.\quad\label{eq: uni-bounded ergodic-1}
\end{equation}
Moreover, for some $\sigma\in\left(0,1\right)$, if $f_{\alpha}\in C^{0,\sigma}([0,\ell_{\alpha}])$  for all $\alpha\in\mathcal{A}$,
then $\left(v,\rho\right)\in C^{2,\sigma}\left(\Gamma\right)\times\mathbb{R}$; there exists a constant $\overline{C}$ which only depends upon $\left\Vert f_{\alpha}\right\Vert _{C^{0,\sigma}\left([0,\ell_{\alpha}]\right)},\mu_{\alpha}$
 and the constants in \eqref{eq: super quadratic} such that 
\begin{equation}
\left\Vert v\right\Vert _{C^{2.\sigma}\left(\Gamma\right)}\le\overline{C}.\label{eq: uni-bounded ergodic}
\end{equation}
\end{thm}

\begin{proof}
[Proof of existence in Theorem \ref{thm: existence ergodic}]
By Proposition \ref{prop: main},
for any $\lambda>0$, the following boundary value problem
\begin{equation}
\begin{cases}
-\mu_{\alpha}\partial^{2}v+H\left(x,\partial v\right)+\lambda v=f, & \text{in }\Gamma_{\alpha}\backslash\mathcal{V},\alpha\in\mathcal{A},\\
v|_{\Gamma_\alpha} (\nu_i) =v|_{\Gamma_\beta} (\nu_i),\quad &\alpha,\beta \in \cA_i, \;i\in I,\\
{\displaystyle \sum_{\alpha\in\mathcal{A}_{i}}\gamma_{i\alpha}\mu_{\alpha}\partial_{\alpha}v\left(\nu_{i}\right)=0}, & i\in I,
\end{cases}\label{eq: ergodic 2}
\end{equation}
has a unique solution $v_{\lambda}\in C^{2}\left(\Gamma\right)$.
Set $C:=\max_{\Gamma}\left|f\left(\cdot\right)-H\left(\cdot,0\right)\right|$.
The constant functions $\varphi:=- C/\lambda$ and $\overline{\varphi}=C/\lambda$
are respectively sub- and super-solution of \eqref{eq: ergodic 2}.
By Lemma \ref{lem: maximum principle},
\begin{equation}
-C\le\lambda v_{\lambda}\left(x\right)\le C,\quad\text{for all }x\in\Gamma.\label{ineq: v is bounded in L_infty}
\end{equation}
Next, set $u_{\lambda}:=v_{\lambda}-\min_{\Gamma}v_{\lambda}$. We see that
$u_{\lambda}$  is the unique classical
solution of
\begin{equation}
\begin{cases}
-\mu_{\alpha}\partial^{2}u_{\lambda}+H\left(x,\partial u_{\lambda}\right)+\lambda u_{\lambda}+\lambda\min_{\Gamma}v_{\lambda}=f, & \text{in }\Gamma_{\alpha}\backslash\mathcal{V},\alpha\in\mathcal{A},\\
u|_{\Gamma_\alpha} (\nu_i) =u|_{\Gamma_\beta} (\nu_i),\quad &\alpha,\beta \in \cA_i, \;i\in I,\\
{\displaystyle \sum_{\alpha\in\mathcal{A}_{i}}\gamma_{i\alpha}\mu_{\alpha}\partial_{\alpha}u_{\lambda}\left(\nu_{i}\right)=0}, & i\in I.
\end{cases}\label{eq: u_lambda}
\end{equation}
Before passing to the limit as $\lambda$ tends $0$,  we need to estimate
$u_{\lambda}$ in $C^{2}\left(\Gamma\right)$ uniformly with respect to $\lambda$. We do this in two steps:

\emph{Step 1: Estimate of $\|\partial u_{\lambda}\|_{L^{q}\left(\Gamma\right)}$}. Using $\psi$ as a test-function in \eqref{eq: u_lambda}, see Definition \ref{def: test function}, and  recalling that  $\lambda u_{\lambda}+\lambda\min_{\Gamma}v_{\lambda}=\lambda v_{\lambda}$,
we see that
\[
\sum_{\alpha\in\mathcal{A}}\int_{\Gamma_{\alpha}}\mu_{\alpha}\partial u_{\lambda}\partial\psi dx+\int_{\Gamma}\left(H\left(x,\partial u_{\lambda}\right)+\lambda v_{\lambda}\right)\psi dx=\int_{\Gamma}f\psi dx.
\]
From \eqref{eq: super quadratic} and \eqref{ineq: v is bounded in L_infty},
\[
\sum_{\alpha\in\mathcal{A}}\int_{\Gamma_{\alpha}}\mu_{\alpha}\partial u_{\lambda}\partial\psi dx+\sum_{\alpha\in\mathcal{A}}\int_{\Gamma_{\alpha}}C_0\left|\partial u_{\lambda}\right|^{q}\psi dx\le\int_{\Gamma}\left(f+C+C_1\right)\psi dx.
\]
On the other hand, since $q>1$, $\psi\ge\underline{\psi}>0$ and $\partial \psi$ is bounded, 
there exists a large enough positive constant $C'$  such that
\[
\sum_{\alpha\in\mathcal{A}}\int_{\Gamma_{\alpha}}\mu_{\alpha}\partial u_{\lambda}\partial\psi dx+\dfrac{1}{2}\sum_{\alpha\in\mathcal{A}}\int_{\Gamma_{\alpha}}C_0\left|\partial u_{\lambda}\right|^{q}\psi dx+C'>0,\quad\text{for all \ensuremath{\lambda>0}}.
\]
Subtracting, we get $
\dfrac{C_0}{2}\underline{\psi}\int_{\Gamma}\left|\partial u_{\lambda}\right|^{q}dx\le\int_{\Gamma}\left(f+C+C_1\right)\psi dx+C'$.
Hence, for all $\lambda>0$, 
\begin{equation}
\left\Vert \partial u_{\lambda}\right\Vert _{L^{q}\left(\Gamma\right)}\le \widetilde {C},\label{ineq: bounded by C1}
\end{equation}
where $\widetilde {C}:=\left[\left(2\int_{\Gamma}\left(|f|+C+C_1\right)\psi dx+2C'\right)/(C_0 \underline{\psi})\right]^{1/q}$.

\emph{Step 2: Estimate of $\|u_{\lambda}\|_{C^{2}\left(\Gamma\right)}$}. Since $u_{\lambda}=v_{\lambda}-\min_{\Gamma}v_{\lambda}$, there exists $\alpha\in \cA$ and 
$x_{\lambda}\in\Gamma_{\alpha}$ such that $u_{\lambda}\left(x_{\lambda}\right)=0$.
For all $\lambda>0$ and $x\in\Gamma_{\alpha}$, we have
\[
\left|u_{\lambda}\left(x\right)\right|=\left|u_{\lambda}\left(x\right)-u_{\lambda}\left(x_{\lambda}\right)\right|\le\int_{\Gamma}\left|\partial u_{\lambda}\right|dx\le\left\Vert \partial u_{\lambda}\right\Vert _{L^{q}\left(\Gamma\right)}\left|\Gamma\right|^{q/\left(q-1\right)}.
\]
From \eqref{ineq: bounded by C1} and the latter inequality, we deduce that
$
\left\Vert u_{\lambda}|_{\Gamma_{\alpha}}\right\Vert _{L^\infty\left(\Gamma_{\alpha}\right)}\le \widetilde{C}\left|\Gamma\right|^{q/\left(q-1\right)}$.
Let $\nu_{i}$ be a transition vertex which belongs to $\partial\Gamma_{\alpha}$.
 For all $\beta\in\mathcal{A}_{i}$,  $y\in\Gamma_{\beta}$,
\[
\left|u_{\lambda}\left(y\right)\right|\le\left|u_{\lambda}\left(y\right)-u_{\lambda}\left(\nu_{i}\right)\right|+\left|u_{\lambda}\left(\nu_{i}\right)\right|\le2\widetilde{C}\left|\Gamma\right|^{q/\left(q-1\right)}.
\]
Since  the network is connected and the number of edges is finite, 
repeating the argument as many times as necessary, we obtain that
there exists $M\in\mathbb{N}$ such that
\[
\left\Vert u_{\lambda}\right\Vert _{L^{\infty}\left(\Gamma\right)}\le M\widetilde{C}\left|\Gamma\right|^{q/\left(q-1\right)}.
\]
This bound is uniform with respect to $\lambda\in (0,1]$.
Next, from \eqref{eq: u_lambda} and (\ref{eq: sup quadractic}),
we get
\[
\underline{\mu}\left|\partial^{2}u_{\lambda}\right|\le\left|H\left(x,\partial u_{\lambda}\right)\right|+\left|\lambda v_{\lambda}\right|+\left|f\right|\le C_{q}\left(1+\left|\partial u_{\lambda}\right|^{q}\right)+C+\left\Vert f\right\Vert _{L^{\infty}\left(\Gamma\right)}.
\]
Hence, from \eqref{ineq: bounded by C1}, $\partial^{2}u_{\lambda} $
is bounded in $L^{1}\left(\Gamma\right)$ uniformly with respect to
$\lambda\in (0,1]$. From the continuous embedding of  $W^{1,1}\left(0,\ell_{\alpha}\right)$ in $C([0,\ell_{\alpha}])$, we infer
that  $\partial u_{\lambda}|_{\Gamma_\alpha} $ is bounded in $C(\Gamma_\alpha)$ uniformly with respect to $\lambda\in (0,1]$. From the equation (\ref{eq: u_lambda}) and (\ref{ineq: v is bounded in L_infty}), this implies that   $u_{\lambda}$  is bounded  in  $C^{2}\left(\Gamma\right)$
 uniformly with respect to $\lambda\in (0,1]$.

After the extraction of a subsequence, we may assume that when $\lambda\rightarrow0^{+}$,
the sequence $ u_{\lambda} $ converges to some function
$v\in C^{1,1}\left(\Gamma\right)$  and that $ \lambda\min v_{\lambda} $
converges to some constant $\rho$. Notice that $v$ still satisfies the Kirchhoff conditions 
since $\partial u_{\lambda}|_{\Gamma_{\alpha}}\left(\nu_{i}\right)\rightarrow\partial v|_{\Gamma_{\alpha}}\left(\nu_{i}\right)$
as $\lambda\rightarrow0^{+}$. 
%From Step 2 and \eqref{ineq: v is bounded in L_infty}, we obtain \eqref{eq: uni-bounded ergodic}.
 Passing to the limit in \eqref{eq: u_lambda}, we get that the couple $\left(v,\rho\right)$
satisfies \eqref{eq: ergodic} in the weak sense, then in the classical sense by using an argument similar to Remark~\ref{rem: regularity for HJ}. Adding a constant to $v$, we also get \eqref{eq: normalization}.

Furthermore,  if  for some $\sigma\in (0,1)$, $f|_{\Gamma_{\alpha}}\in C^{0,\sigma}\left(\Gamma_{\alpha}\right)$
for all $\alpha\in\mathcal{A}$, a bootstrap argument using the Lipschitz continuity of $H$ on the bounded subsets
 of $\Gamma\times \R$ shows   that $u_{\lambda}$  is bounded  in  $C^{2,\sigma}\left(\Gamma\right)$
 uniformly with respect to $\lambda\in (0,1]$. After a further extraction of a subsequence if necessary,
 we obtain  (\ref{eq: uni-bounded ergodic}).
\end{proof}
\begin{proof}[Proof of uniqueness in Theorem \ref{thm: existence ergodic}] Assume that there exist
two solutions $\left(v,\rho\right)$ and $\left(\tilde{v},\tilde{\rho}\right)$
of \eqref{eq: ergodic}-\eqref{eq: normalization}. First of all, we claim that $\rho=\tilde{\rho}$. By symmetry,  it suffices to prove that
$\rho \ge\tilde{\rho}$. Let $x_{0}$ be a maximum point of $e:=\tilde{v}-v$. Using similar arguments as in the proof of Lemma \ref{lem: maximum principle}, with
$\lambda v$ and $\lambda u $  respectively replaced by $\rho$ and $\tilde{\rho}$, we get 
$\rho\ge\tilde{\rho}$ and the claim is proved.

We now prove the uniqueness
of $v$. Since $H_{\alpha}$ belongs to $C^{1}\left(\Gamma_{\alpha}\times\mathbb{R}\right)$
for all $\alpha\in\mathcal{A}$, then $e$ is a solution of 
$
\mu_{\alpha}\partial^{2}e_{\alpha}-\left[\int_{0}^{1} \partial_p H_{\alpha}\left(y,\theta\partial v_{\alpha}+\left(1-\theta\right)\partial \tilde{v}_{\alpha}\right)d\theta\right]\partial e_{\alpha}=0,\quad\text{in }(0,\ell_{\alpha})$,
with the same transmission and boundary condition as in \eqref{eq: ergodic}.
By Lemma \ref{lem: linear 2}, $e$ is a constant function on $\Gamma$.
Moreover, from  (\ref{eq: normalization}) , we know that $\int_{\Gamma}edx=0$.
This yields that $e=0$ on $\Gamma$. Hence, \eqref{eq: ergodic}-\eqref{eq: normalization}
has a unique solution.  
\end{proof} 
\begin{rem}
Since there exists a unique solution of \eqref{eq: ergodic}-\eqref{eq: normalization},
we conclude that the whole sequence $\left(u_{\lambda},\lambda v_{\lambda}\right)$
in the proof of Theorem \ref{thm: existence ergodic} converges to
$\left(v,\rho\right)$ as $\lambda\rightarrow0$.
\end{rem}

\section{Proof of the main result}

We first prove Theorem~\ref{thm: MFG system} when $F$ is bounded.
\begin{thm}
\label{thm: main result with F bounded} Assume \eqref{eq: Hamiltonian}-(\ref{eq: derivative of H is bounded}), \eqref{eq: coupling F}
and that $F$ is bounded. There exists a  solution
$\left(v,m,\rho\right)\in C^{2}\left(\Gamma\right)\times\mathcal{W}\times\mathbb{R}$
to the mean field games system \eqref{eq: MFG system}. 
If $F$ is locally Lipschitz continuous, then $v\in C^{2,1}(\Gamma)$.
If furthermore $F$ is strictly increasing, then the solution is unique.
\end{thm}

\begin{proof}
[Proof of existence in Theorem \ref{thm: main result with F bounded}]We
adapt the proof of Camilli and Marchi in \cite[Theorem 1]{CM2016}.
For $\sigma\in (0,1/2)$ let us introduce the space
\[
\mathcal{M}_{\sigma}=\left\{ m:m_{\alpha}\in C^{0,\sigma}\left(\left[0,\ell_{\alpha}\right]\right)\text{  and }\dfrac{m|_{\Gamma_{\alpha}}\left(\nu_{i}\right)}{\gamma_{i\alpha}}=\dfrac{m|_{\Gamma_{\beta}}\left(\nu_{i}\right)}{\gamma_{i\beta}}\text{ for all \ensuremath{i\in I} and \ensuremath{\alpha,\beta\in\mathcal{A}_{i}}}\right\} 
\]
which, endowed with the norm $ {\displaystyle
\left\Vert m\right\Vert _{\mathcal{M}_{\sigma}}=\left\Vert m\right\Vert _{L^{\infty}\left(\Gamma\right)}+\max_{\alpha\in\mathcal{A}}\sup_{y,z\in\left[0,\ell_{\alpha}\right],y\ne z}\dfrac{\left|m_{\alpha}\left(y\right)-m_{\alpha}\left(z\right)\right|}{\left|y-z\right|^{\sigma}}}$,
is a Banach space. Now consider the set
\[
\mathcal{K}=\left\{ m\in\mathcal{M}_{\sigma}:m\ge0\text{ and }\int_{\Gamma}mdx=1\right\} 
\]
and observe that $\mathcal{K}$ is a closed and convex subset of $\mathcal{M}_{\sigma}$. We define a map $T:\mathcal{K}\rightarrow\mathcal{K}$
as follows: given $m\in\mathcal{K}$, set $f=F\left(m\right)$. By
Theorem \ref{thm: existence ergodic},
\eqref{eq: ergodic}-\eqref{eq: normalization} has a unique solution
$\left(v,\rho\right)\in C^{2}\left(\Gamma\right)\times\mathbb{R}$.
Next, for $v$ given, we solve \eqref{eq: Fokker-Planck}-\eqref{eq: normalization and positive}
with $b\left(\cdot\right)=\partial_{p}H\left(\cdot,\partial v\left(\cdot\right)\right)\in PC(\Gamma)$.
By Theorem \ref{thm: existence and uniqueness FK}, there exists a
unique solution $\overline{m}\in\mathcal{K}\cap W$ of \eqref{eq: Fokker-Planck}-\eqref{eq: normalization and positive}.
We set $T\left(m\right)=\overline{m}$; we claim that $T$ is 
continuous and has a precompact image. We proceed in several steps:

\emph{$T$ is continuous}. Let $m_{n},m\in\mathcal{K}$ be such that $\left\Vert m_{n}-m\right\Vert _{\mathcal{M}_{\sigma}}\rightarrow0$
as $n\rightarrow+\infty$; set $\overline{m}_{n}=T\left(m_{n}\right),\overline{m}=T\left(m\right)$.
We need to prove that $\overline{m}_{n}\rightarrow\overline{m}$ in
$\mathcal{M}_{\sigma}$. Let $\left(v_{n},\rho_{n}\right),\left(v,\rho\right)$
be the solutions of \eqref{eq: ergodic}-\eqref{eq: normalization}
corresponding respectively to $f=F\left(m_{n}\right)$ and $f=F\left(m\right)$. Using estimate (\ref{eq: uni-bounded ergodic-1}), we see that up
to the extraction of a subsequence, we may assume that $\left(v_{n},\rho_{n}\right)\rightarrow\left(\overline{v},\overline{\rho}\right)$
in $C^{1}\left(\Gamma\right)\times\mathbb{R}$.  Since 
$F\left(m_{n}\right)|_{\Gamma_\alpha}\rightarrow F\left(m\right)|_{\Gamma_\alpha}$
in $C\left(\Gamma_\alpha\right)$, $H_{\alpha}\left(y,(\partial v_{n})_{\alpha}\right)\rightarrow H_{\alpha}\left(y,\partial\overline{v}_{\alpha}\right)$
in $C\left([0,\ell_{\alpha}]\right)$,
and since it is possible to pass to the limit in the transmission and boundary conditions thanks to the $C^{1}$-convergence, 
we obtain that $\left(\overline{v},\overline{\rho}\right)$ is a weak (and strong by Remark \ref{rem: regularity for HJ}) 
solution of \eqref{eq: ergodic}-\eqref{eq: normalization}.
By uniqueness, $(\overline{v},\overline{\rho})=(v,\rho)$ and the whole sequence $\left(v_{n},\rho_{n}\right)$ converges.

Next, $\overline{m}_{n}=T\left(m_{n}\right),\overline{m}=T\left(m\right)$
are respectively the solutions
of \eqref{eq: Fokker-Planck}-\eqref{eq: normalization and positive}
corresponding to $b=\partial_{p}H\left(x,\partial v_{n}\right)$ and
$b=\partial_{p}H\left(x,\partial v\right)$. From
the estimate \eqref{eq: estimate FK}, since $\partial_{p}H\left(x,\partial v_{n}\right) $
is uniformly bounded in $L^{\infty}\left(\Gamma\right)$, we see
that $\overline{m}_{n}$ is uniformly bounded in
$W$. Therefore, up to the extraction of subsequence, $\overline{m}_{n}\rightharpoonup\widehat{m}$ in $W$ and $\overline{m}_{n}\rightarrow\widehat{m}$ in $\mathcal{M}_\sigma$,
% \[
% \begin{cases}
% \overline{m}_{n}\rightharpoonup\widehat{m} & \text{in \ensuremath{W}},\\
% \overline{m}_{n}\rightarrow\widehat{m} & \hbox{ in }\mathcal{M}_\sigma,
% \end{cases}
% \]
because $W$ is compactly embedded in $\mathcal{M}_\sigma$ for $\sigma\in (0,1/2)$. It is easy to pass to the limit and find that  $\widehat{m}$ is a solution
of \eqref{eq: Fokker-Planck}-\eqref{eq: normalization and positive} with $b=\partial_{p}H\left(x,\partial v\right)$.
From Theorem \ref{thm: existence and uniqueness FK},
we obtain that $\overline{m}=\widehat{m}$, and hence the whole sequence $\overline{m}_{n} $
converges to $\overline{m}$.

\emph{The image of $T$ is precompact}.  Since  $F\in C^{0}\left(\mathbb{R}^{+};\mathbb{R}\right)$ is a uniformly bounded function,
we see that  $F\left(m\right)$ is bounded in $L^{\infty}\left(\Gamma\right)$
uniformly with respect to $m\in \mathcal{K}$. From Theorem \ref{thm: existence ergodic},
there exists a constant $\overline{C}$ such that for all $m\in\mathcal{K}$, 
the unique solution $v$ of \eqref{eq: ergodic}-\eqref{eq: normalization} with $f=F(m)$ satisfies 
$\left\Vert v\right\Vert _{C^{2}\left(\Gamma\right)}\le\overline{C}$.
From Theorem \ref{thm: existence and uniqueness FK}, we obtain that
$\overline{m}=T\left(m\right)$ is bounded in $W$ by a constant independent
of $m$. Since $W$ is compactly embedded in $\mathcal{M}_\sigma$, for $\sigma\in (0,1/2)$ we deduce
that $T$ has a precompact image.
% \emph{$T$ is a compact mapping}. Consider a sequence $\left\{ m_{n}\right\} $ such that $\left\Vert m_{n}\right\Vert _{\mathcal{M}}\le1$.
% Set $\overline{m}_{n}=T\left(m_{n}\right)$ and $\left(v_{n},\rho_{n}\right)$
% be the unique solution of \eqref{eq: ergodic}-\eqref{eq: normalization}
% with $f=F\left(m_{n}\right)$. From $\eqref{eq: uni-bounded ergodic}$,
% we have that $\left\{ v_{n}\right\} $ is uniformly bounded in $C^{2}\left(\Gamma\right)$
% and thus $\partial_{p}H\left(x,\partial v_{n}\right)$ is uniformly
% bounded in $L^{\infty}\left(\Gamma\right)$. Therefore, from \eqref{eq: estimate FK}
% and the compact embedding of $W$ into $\mathcal{M}_{\sigma}$, we
% also have that $\left\{ m_{n}\right\} $ is compact in $\mathcal{M}_{\sigma}$.
% Hence, $T$ is compact.

% \emph{$T$ has a compact image}. Since the local coupling $F\in C^{1}\left(\mathbb{R}^{+};\mathbb{R}\right)$ is bounded,
%  $F\left(m\right)$ is bounded in $L^{\infty}\left(\Gamma\right)$
% with a constant independent of $m\in \mathcal{K}$. From Theorem \ref{thm: existence ergodic},
% there exists a constant $\overline{C}$ such that for all $m\in\mathcal{K}$, 
% the unique solution $v$ of \eqref{eq: ergodic}-\eqref{eq: normalization} with $f=F(m)$ satisfies 
% $\left\Vert v\right\Vert _{C^{2}\left(\Gamma\right)}\le\overline{C}$.
% From Theorem \ref{thm: existence and uniqueness FK}, we obtain that
% $\overline{m}=T\left(m\right)$ is bounded in $W$ by a constant independent
% of $m$. Since $W$ is compactly embedded in $\mathcal{M}$, we deduce
% that $T$ has a compact image.

\emph{End of the proof}. We can apply  Schauder
fixed point theorem (see \cite[Corollary 11.2]{GT2001}) to conclude  that the
map $T$ admits a fixed point $m$. By Theorem \ref{thm: existence and uniqueness FK}, we get $m\in \mathcal{W}$. Hence, there exists a solution $(v,m,\rho)\in C^2(\Gamma)\times \mathcal{W}\times \mathbb{R}$ to the mean field games system \eqref{eq: MFG system}. If $F$ is locally Lipschitz continuous, then $v\in C^{2, 1}(\Gamma)$ 
from the final part of Theorem \ref{thm: existence ergodic}. 
\end{proof}

\begin{proof}
[Proof of uniqueness in Theorem \ref{thm: main result with F bounded}]We
assume that $F$ is strictly increasing and that there exist two solutions $\left(v_{1},m_{1},\rho_{1}\right)$
and $\left(v_{2},m_{2},\rho_{2}\right)$ of \eqref{eq: MFG system}. We set $\overline{v}=v_{1}-v_{2},\overline{m}=m_{1}-m_{2}$ and $\overline{\rho}=\rho_{1}-\rho_{2}$
and write the equations for $\overline{v},\overline{m}$ and $\overline{\rho}$
\begin{equation}
\begin{cases}
-\mu_{\alpha}\partial^{2}\overline{v}+H\left(x,\partial v_{1}\right)-H\left(x,\partial v_{2}\right)+\overline{\rho}-\left(F\left(m_{1}\right)-F\left(m_{2}\right)\right)=0, & \text{in \ensuremath{\Gamma_{\alpha}\backslash\mathcal{V}}},\\
-\mu_{\alpha}\partial^{2}\overline{m}-\partial\left(m_{1}\partial_{p}H\left(x,\partial v_{1}\right)\right) +  \partial\left( m_{2}\partial_{p}H\left(x,\partial v_{2}\right)\right)=0, &
 \text{in \ensuremath{\Gamma_{\alpha}\backslash\mathcal{V}}},\\
\overline{v}|_{\Gamma_{\alpha}}\left(\nu_{i}\right)=\overline{v}|_{\Gamma_{\beta}}\left(\nu_{i}\right),\quad
\dfrac{\overline{m}|_{\Gamma_{\alpha}}\left(\nu_{i}\right)}{\gamma_{i\alpha}}=\dfrac{\overline{m}|_{\Gamma_{\beta}}\left(\nu_{i}\right)}{\gamma_{i\beta}}, &
 \alpha,\beta\in\mathcal{A}_{i}, i\in I, \\
{\displaystyle \sum_{\alpha\in\mathcal{A}_{i}}\gamma_{i\alpha}\mu_{\alpha}\partial_{\alpha}\overline{v}\left(\nu_{i}\right)}=0, & i\in I,\\
{\displaystyle \sum_{\alpha\in\mathcal{A}_{i}}   n_{i\alpha} \left[m_{1}|_{\Gamma_{\alpha}}\left(\nu_{i}\right)\partial_{p}H\left(\nu_{i},\partial v_{1}|_{\Gamma_\alpha} \left(\nu_{i}\right)\right)-m_{2}|_{\Gamma_{\alpha}}\left(\nu_{i}\right)\partial_{p}H\left(\nu_{i},\partial v_{2}|_{\Gamma_\alpha}\left(\nu_{i}\right)\right)\right]}\\
+{\displaystyle \sum_{\alpha\in\mathcal{A}_{i}}\mu_{\alpha}\partial_{\alpha}\overline{m}\left(\nu_{i}\right)}=0, & i\in I,\\
\ds \int_{\Gamma}\overline{v}dx=0,\quad \int_{\Gamma}\overline{m}dx=0.
\end{cases}
\end{equation}
Multiplying the equation for $\overline{v}$ by $\overline{m}$ and
integrating over $\Gamma_{\alpha}$, we get
\begin{equation}
\label{eq:20}
   \int_{\Gamma_{\alpha}}\mu_{\alpha}\partial\overline{v}\partial\overline{m}+\left[H\left(x,\partial v_{1}\right)-H\left(x,\partial v_{2}\right)+\overline{\rho}-\left(F\left(m_{1}\right)-F\left(m_{2}\right)\right)\right]\overline{m}dx
 -\left[\mu_{\alpha}\overline{m}_{\alpha}\partial\overline{v}_{\alpha}\right]_{0}^{\ell_{\alpha}}=0.
\end{equation}
Multiplying the equation for $\overline{m}$ by $\overline{v}$ and
integrating over $\Gamma_{\alpha}$, we get
\begin{align}
 & \int_{\Gamma_{\alpha}}\mu_{\alpha}\partial\overline{v}\partial\overline{m}+\left[m_{1}\partial_{p}H\left(x,\partial v_{1}\right)-m_{2}\partial_{p}H\left(x,\partial v_{2}\right)\right]\partial\overline{v}dx\label{eq: multiplying v}\\
 & -\Bigl[\overline{v}|_{\Gamma_\alpha}\left(\mu_{\alpha}\partial\overline{m}|_{\Gamma_\alpha}+m_{1}|_{\Gamma_\alpha}\partial_{p}H\left(x,\partial v_{1}|_{\Gamma_\alpha}\right)-m_{2}|_{\Gamma_\alpha}\partial_{p}H\left(x,\partial v_{2}|_{\Gamma_\alpha}\right)\right)\Bigr]_{0}^{\ell_{\alpha}}=0.\nonumber 
\end{align}
Subtracting \eqref{eq:20} to \eqref{eq: multiplying v},
summing over $\alpha\in\mathcal{A}$, assembling the terms corresponding
to a same vertex $\nu_{i}$ and taking into account the transmission
and the normalization condition for $\overline{v}$ and $\overline{m}$,
we obtain
\begin{align*}
0=
 & \sum_{\alpha\in\mathcal{A}}\int_{\Gamma_{\alpha}}\left(m_{1}-m_{2}\right)\left[F\left(m_{1}\right)-F\left(m_{2}\right)\right]dx\\
 & +\sum_{\alpha\in\mathcal{A}}\int_{\Gamma_{\alpha}}m_{1}\left[H\left(x,\partial v_{2}\right)-H\left(x,\partial v_{1}\right)+ \partial_{p}H\left(x,\partial v_{1}\right)\partial\overline{v}\right]dx\\
 & +\sum_{\alpha\in\mathcal{A}}\int_{\Gamma_{\alpha}}m_{2}\left[H\left(x,\partial v_{1}\right)-H\left(x,\partial v_{2}\right)-\partial_{p}H\left(x,\partial v_{2}\right)\partial\overline{v}\right]dx.
\end{align*}
Since $F$ is strictly monotone then the first sum is non-negative.
Moreover, by the convexity of $H$ and the positivity of $m_{1},m_{2}$,
the last two sums are non-negative. Therefore, we have that $m_{1}=m_{2}$.
From Theorem \ref{thm: existence ergodic}, we finally obtain $v_{1}=v_{2}$
and $\rho_{1}=\rho_{2}$.
\end{proof}

\begin{proof}
[Proof of Theorem \ref{thm: MFG system} for a general coupling $F$] We only need to modify the proof of existence.

We now truncate the coupling function as follows:
\[
F_{n}\left(r\right)=\begin{cases}
F\left(r\right), & \text{if }\left|r\right|\le n,\\
F\left(\dfrac{r}{\left|r\right|}n\right), & \text{if }\left|r\right|\ge n.
\end{cases}
\]
Then $F_{n}$ is continuous, bounded below by $-M$ as in \eqref{ineq: F is bounded from below} and bounded above by some constant $C_n$. By Theorem \ref{thm: main result with F bounded},
for all $n\in\mathbb{N}$, there exists a unique solution $\left(v_{n},m_{n},\rho_{n}\right)\in C^{2}\left(\Gamma\right)\times\mathcal{W}\times\mathbb{R}$
of the mean field game system \eqref{eq: MFG system} where $F$ is
replaced by $F_{n}$. We wish to pass to the limit as $n\to +\infty$. We proceed in  several steps:

\emph{Step 1:  $\rho_{n}$ is bounded from below}. Multiplying the HJB equation in \eqref{eq: MFG system}
by $m_{n}$ and the Fokker-Planck equation in \eqref{eq: MFG system}
by $v_{n}$, using integration by parts and the transmission conditions, we obtain that
\begin{equation}
\sum_{\alpha\in\mathcal{A}}\int_{\Gamma_{\alpha}}\mu_{\alpha}\partial v_{n}\partial m_{n}dx+\int_{\Gamma}H\left(x,\partial v_{n}\right)m_{n}dx+\rho_{n}=\int_{\Gamma}F_{n}\left(m_{n}\right)m_{n}dx,\label{eq: HJn}
\end{equation}
and
\begin{equation}
\sum_{\alpha\in\mathcal{A}}\int_{\Gamma_{\alpha}}\mu_{\alpha}\partial v_{n}\partial m_{n}dx+\int_{\Gamma}\partial_{p}H\left(x,\partial v_{n}\right)m_{n}\partial v_{n}dx=0.\label{eq: FKn}
\end{equation}
Subtracting the two equations, we obtain
\begin{equation}
\rho_{n}=\int_{\Gamma}F_{n}\left(m_{n}\right)m_{n}dx+\int_{\Gamma}\left[\partial_{p}H\left(x,\partial v_{n}\right)\partial v_{n}-H\left(x,\partial v_{n}\right)\right]m_{n}dx.\label{eq: formula rho_n}
\end{equation}
In what follows, the constant $C$ may vary from line to line but remains independent of $n$. From (\ref{eq: Hamiltonian is convex}), we see that $\partial_{p}H\left(x,\partial v_{n}\right)\partial v_{n}-H\left(x,\partial v_{n}\right)\ge -H(x,0)\ge -C$.  Therefore
\begin{equation}
  \label{eq:21}
\rho_{n}  \ge\int_{\Gamma}F_{n}\left(m_{n}\right)m_{n}dx - C \int_{\Gamma} m_n dx
  =\int_{\Gamma}F_{n}\left(m_{n}\right)m_{n}dx-C.
\end{equation}
Hence, since $F_n+M\ge0$ and $\int_{\Gamma}m_n dx=1$, we get that $\rho_{n}$ is bounded from below by $-M-C$ independently of $n$.

\emph{Step 2: $\rho_{n} $ and $\int_{\Gamma}F_{n}\left(m_{n}\right)dx$
are uniformly bounded}. By Theorem \ref{thm: existence and uniqueness FK}, there exists a
positive solution $w\in W$ of \eqref{eq: Fokker-Planck}-\eqref{eq: normalization and positive} with $b=0$.
It yields
\begin{displaymath}
{\displaystyle \sum_{\alpha\in\mathcal{A}}
\int_{\Gamma_{\alpha}}\mu_{\alpha}\partial w\partial udx=0},  \text{for all }u\in V,\quad\quad \hbox{and}\quad 
  \int_{\Gamma}wdx=1.
\end{displaymath}
% \[
% \begin{cases}
% {\displaystyle \sum_{\alpha\in\mathcal{A}}\int_{\Gamma_{\alpha}}\mu_{\alpha}\partial w\partial udx=0}, & \text{for all }u\in V,\\
% \int_{\Gamma}wdx=1.
% \end{cases}
% \]
Multiplying the HJB
equation of \eqref{eq: MFG system} by $w$, using integration by
parts and the Kirchhoff condition, we get
\[
\underset{=0}{\underbrace{\sum_{\alpha\in\mathcal{A}}\int_{\Gamma_{\alpha}}\mu_{\alpha}\partial v_{n}\partial wdx}}+\int_{\Gamma}H\left(x,\partial v_{n}\right)wdx+\rho_{n}\underset{=1}{\underbrace{\int_{\Gamma}wdx}}=\int_{\Gamma}F_{n}\left(m_{n}\right)wdx.
\]
This implies, using \eqref{eq: super quadratic}, \eqref{eq: estimate FK} and $F_n+M\ge0$,
\begin{align}
\rho_{n} & =\int_{\Gamma}F_{n}\left(m_{n}\right)wdx-\int_{\Gamma}H\left(x,\partial v_{n}\right)wdx \nonumber\\
 & \le\left\Vert w\right\Vert _{L^{\infty}\left(\Gamma\right)}\int_{\Gamma}\left(F_{n}(m_{n})+M\right)dx-M-\int_{\Gamma}\left(C_{0}\left|\partial v_{n}\right|^{q}-C_{1}\right)wdx \nonumber\\
 & \le C\int_{\Gamma}F_{n}\left(m_{n}\right)dx+C-\int_{\Gamma}C_{0}\left|\partial v_{n}\right|^{q}wdx.\label{ineq: rho_n is bounded from above}
\end{align}
Thus, by \eqref{eq:21}, we have
\begin{equation}
-M-C\le\int_{\Gamma}F_{n}\left(m_{n}\right)m_{n}dx-C\le\rho_{n}\le C\int_{\Gamma}F_{n}\left(m_{n}\right)dx+C.\label{ineq: rho, F(m) and F(m)m}
\end{equation}
Let $K>0$ be a constant to be chosen later. We have
\begin{align}
\int_{\Gamma}F_{n}\left(m_{n}\right)dx 
% & \le\int_{m_{n}\ge K}(F_{n}\left(m_{n}\right)+M)dx+\int_{m_{n}\le K}F_{n}\left(m_{n}\right)dx\nonumber \\
% & \le\dfrac{1}{K}\int_{m_{n}\ge K}\left[F_{n}\left(m_{n}\right)+M\right]m_{n}dx+\sup_{0\le r\le K}F_{n}(r)\int_{m_{n}\le K}dx\nonumber \\
 & \le\dfrac{1}{K}\int_{m_{n}\ge K}\left[F_{n}\left(m_{n}\right)+M\right]m_{n}dx+\sup_{0\le r\le K}F (r)\int_{m_{n}\le K}dx\nonumber \\
 & \le\dfrac{1}{K}\int_{\Gamma}F_{n}\left(m_{n}\right)m_{n}dx+\dfrac{M}{K}+C_K,\label{ineq: F(m)m and F(m)}
\end{align}
where $C_K$ is independent of $n$. Choosing $K=2C$ where $C$ is the constant in \eqref{ineq: rho, F(m) and F(m)m}, we get by combining 
(\ref{ineq: F(m)m and F(m)}) with  \eqref{ineq: rho, F(m) and F(m)m} that 
$\int_{\Gamma}F_n(m_n)m_n\le C$.
Using \eqref{ineq: F(m)m and F(m)} again, we obtain 
$
\int_{\Gamma}F_n(m_n)dx\le C$.
Hence, from \eqref{ineq: rho, F(m) and F(m)m}, we conclude that  $|\rho_n|+ |\int_{\Gamma}F_n(m_n)dx|\le C$.

\emph{Step 3: Prove that $ F_{n}\left(m_{n}\right) $
is uniformly integrable and $v_{n} $ and $ m_{n} $
are uniformly bounded  respectively in $C^{1}\left(\Gamma\right)$ and $W$}. Let $E$ be a measurable with $|E|=\eta$. By \eqref{ineq: F(m)m and F(m)} with $\Gamma$ is replaced by $E$, we have
\begin{align*}
\int_{E}F_n (m_n)   m_n dx & \le\dfrac{1}{K}\int_{E\cap \left\{m_n\ge K \right\}}F_{n}\left(m_{n}\right)m_{n}dx+\dfrac{M}{K}+\sup_{0\le r\le K}F(r)\int_{E\cap \left\{m_n\le K \right\}}dx\\
 & \le \dfrac{C+M}{K}+C_{K}\eta,
\end{align*}
since $\int_{E} F_n(m_n) m_n dx\le C$ and $\sup_{0\le r \le K}F_{n}(r)\le \sup_{0\le r \le K}F(r):=C_K$. Therefore, for all $\varepsilon>0$, we may
choose $K$ such that $(C+M)/K\le\varepsilon/2$ and then $\eta$ such that $C_K \eta\le\varepsilon/2$  and get
\[
\int_{E}F_{n}\left(m_{n}\right)dx\le \varepsilon,\quad\text{for all \ensuremath{E} which satisfies \ensuremath{\left|E\right|\le\eta}},
\]
which proves the uniform integrability of $\left\{F_n (m_n) \right\}_n$.

Next, since $ \rho_{n} $ and $ \int_{\Gamma}F_{n}\left(m_{n}\right)dx $
are uniformly bounded, we infer from  \eqref{ineq: rho_n is bounded from above} that
$\partial v_{n} $ is uniformly bounded in $L^{q}\left(\Gamma\right)$.
Since  by the condition $\int_{\Gamma}v_ndx=0$, there exists $\overline{x}_n$ such that $v_n(\overline{x}_n)=0$, we infer from the latter bound that 
$v_n$ is uniformly bounded in $L^{\infty}(\Gamma)$. 
%and $\partial v_n$ is bounded in $L^1(\Gamma)$ by a constant independent of $n$.
% Since $v_n$ is $C^2$ in every edge $\Gamma_{\alpha}\ni\nu_i$, for all $x\in \Gamma_{\alpha}$, we have
% \begin{equation}
% |v_n(x)|\le |v_{n}(\nu_i)|+\int_{\pi^{-1}_{\alpha}(\nu_i)}^{\pi^{-1}_{\alpha}(x)}|\partial v_{\alpha}(y)|dy \le |v_n(\nu_i)|+C\left\Vert\partial v \right\Vert_{L^q(\Gamma)}\le |v_n(\nu_i)|+C. \label{ineq: v_n is bounded in L^infty}
% \end{equation}
% Using that $\Gamma$ is of finite length with a finite number of vertices and since, by the condition $\int_{\Gamma}v_ndx=0$, there exists $\overline{x}_n$ such that $v_n(\overline{x}_n)=0$, we infer from \eqref{ineq: v_n is bounded in L^infty} that $v_n$ is bounded in $L^{\infty}(\Gamma)$ and $\partial v_n$ is bounded in $L^1(\Gamma)$ by a constant independent of $n$.
Using the HJB equation in \eqref{eq: MFG system} and Remark \ref{rem: derivative of H is bounded}, we get
\[
\mu_{\alpha}|\partial^2 v_n|\le |H(x,\partial v_n)| +|F_n(m_n)|+|\rho_n| \le C_q (|\partial v_n|^q+1)+|F_n(m_n)|+|\rho_n|.
\]
 We obtain that $ \partial^{2}v_{n} $ is uniformly bounded in $L^{1}\left(\Gamma\right)$, which implies that   $v_n$ is uniformly bounded in $C^1(\Gamma)$.
Therefore the sequence of functions  $ C_q (|\partial v_n|^q+1)+|F_n(m_n)|+|\rho_n|$ is uniformly integrable, and so is  $ \partial^{2}v_{n} $. This implies that $ \partial v_{n} $
is equicontinuous. Hence, $\left\{ v_{n}\right\} $ is relatively
compact in $C^{1}\left(\Gamma\right)$ by Arzel{\`a}-Ascoli's theorem.
Finally, from the Fokker-Planck equation and Theorem \ref{thm: existence and uniqueness FK},
since $ \partial_{p}H\left(x,\partial v_{n}\right) $
is uniformly bounded in $L^{\infty}\left(\Gamma\right)$, we obtain
that $m_{n}$ is uniformly bounded in $W$.

\emph{Step 4: Passage to the limit}

From Step 1 and 2, since $\left\{ \rho_{n}\right\} $ is uniformly bounded,  there exists $\rho\in\mathbb{R}$
such that $\rho_{n}\rightarrow\rho$  up to the extraction
of subsequence. From Step 3, there exists $m\in W$ such that $m_{n}\rightharpoonup m$
in $W$ and $m_{n}\to m$ almost everywhere, up to the extraction
of subsequence. Also from Step 3, since $F_{n}\left(m_{n}\right) $
is uniformly integrable, from Vitali theorem, 
$
\lim_{n\rightarrow\infty}\int_{\Gamma}F_{n}\left(m_{n}\right)\tilde{w}dx=\int_{\Gamma}F\left(m\right)\tilde{w}dx,\quad\text{for all }\tilde{w}\in W$.
From Step 3, up to the extraction of subsequence, there exists $v\in C^{1}\left(\Gamma\right)$
such that $v_{n}\rightarrow v$ in $C^{1}\left(\Gamma\right)$. Hence,
$\left(v,\rho,m\right)$ satisfies the weak form of the MFG system:
\[
\sum_{\alpha\in\mathcal{A}}\int_{\Gamma_{\alpha}}\mu_{\alpha}\partial v\partial\tilde{w}dx+\int_{\Gamma}  \left(H\left(x,\partial v\right)+\rho\right)\tilde{w}dx=\int_{\Gamma}F\left(m\right)\tilde{w}dx,\quad\text{for all }\tilde{w}\in W,
\]
and
\[
\sum_{\alpha\in\mathcal{A}}\int_{\Gamma_{\alpha}}\mu_{\alpha}\partial m\partial\tilde{v}dx+\int_{\Gamma}\partial_{p}H\left(x,\partial v\right)m\partial\tilde{v}dx=0,\quad\text{for all }\tilde{v}\in V.
\]
Finally, we prove the regularity for the solution of~(\ref{eq: MFG system})
. Since $m\in W$,  $m|_{\Gamma_\alpha}\in C^{0,\sigma}$ for some constant $\sigma\in (0,1/2)$ and all $\alpha\in \cA$. By Theorem \ref{thm: existence ergodic}, $v\in C^2(\Gamma)$  ($v\in C^{2,\sigma}(\Gamma)$ if $F$ is locally Lipschitz continuous). Then, by Theorem \ref{thm: existence and uniqueness FK}, we see that $m\in \mathcal{W}$.
 If $F$ is locally Lipschitz continuous, this implies that $v\in C^{2,1}(\Gamma)$.
 We also obtain that $v$ and $m$ satisfy the Kirchhoff and transmission conditions in \eqref{eq: MFG system}. The proof is done.
\end{proof}

\noindent{\bf Acknowledgement.} \quad  
 The  authors were partially  supported by ANR project ANR-16-CE40-0015-01.

%\nocite{Nicaise1998} \nocite{BN1996}

\end{document}